\documentclass[10pt, final]{amsart} 
\usepackage{amsmath,amssymb,amsbsy, amsfonts,amsthm, amscd, amssymb,amscd,mathrsfs, showkeys} 
\usepackage[dvips]{graphicx}
\usepackage{multirow}
\usepackage{epsfig}
\usepackage[subnum]{cases}
\usepackage[colorlinks=true, linkcolor=magenta, citecolor=cyan, urlcolor=blue,hyperfootnotes=false]{hyperref}
\newcommand{\R}{{\mathbb R}}
 
\newcommand{\T}{{\mathbb T}} 

\newcommand{\C}{{\mathbb C}}

\newcommand{\Lin}{\mathcal{L}}

 \renewcommand{\geq }{\geqslant}
 \renewcommand{\leq }{\leqslant}

\usepackage{mathtools}
\DeclarePairedDelimiter{\abs}{\lvert}{\rvert}

\DeclarePairedDelimiter{\norma}{\lVert}{\rVert} 
\usepackage{braket}
\usepackage{color} 

{\left\lbrace\begin{array}{@{}l@{}}}% 
{\end{array}\right.}  
\newcommand{\squad}[1]{{\left [ #1 \right ]}}
\newcommand{\round}[1]{{\left ( #1 \right )}}
\newcommand{\norm}[2]{{\left\| #1 \right\|}_{#2}}

\numberwithin{equation}{section}
\newtheorem{theorem}{Theorem}[section]
\newtheorem{corollary}[theorem]{Corollary}
\newtheorem{lemma}[theorem]{Lemma}
\newtheorem{proposition}[theorem]{Proposition}
\theoremstyle{definition} 
\newtheorem{definition}[theorem]{Definition}

\newtheorem{remark}[theorem]{Remark}

\begin{document} 

\title[$H^2$-scattering for HFC4]{Decay in energy space for the solution of fourth-order Hartree-Fock equations with general non-local interactions}

%%%%%%%%%%%%%%%%%%%%%%%%AUTHORS%%%%%%%%%%%%%%%%%%%%%%%%

\author{M.~Tarulli}
\address{Department of Mathematical Analysis and Differential Equations, Technical University of Sofia, Kliment Ohridski Blvd. 8, 1000 Sofia, Bulgaria.
    Institute of Mathematics and Informatics at Bulgarian Academy of Science, Acad. Georgi Bonchev Str., Block 8, 1113 Sofia, Bulgaria. Dipartimento di Matematica, Universit\`a di Pisa,
Largo B. Pontecorvo 5, 56127 Pisa, Italy}
\email{mta@tu-sofia.bg}

\author{G.~Venkov}
\address{Department of Mathematical Analysis and Differential Equations, Technical University of Sofia, Kliment Ohridski Blvd. 8, 1000 Sofia.}
\email{gvenkov@tu-sofia.bg}


%%%%%%%%%%%%%%%%%%%%%%%%%%%%END OF AUTHORS%%%%%%%%%%%%%

\begin{abstract}
We prove the decay in the energy space for the solution to the defocusing biharmonic Hartree-Fock equations with mass-supercritical and energy-subcritical Choquard-type nonlinearity in space dimension $d\geq3$. We treat both the free and the perturbed by a potential case. As a direct consequence, we obtain large-data scattering in  $H^2(\R^d)^N, \, N\geq 1$.
\end{abstract}

%\date{\today}
\maketitle

\section{Introduction}\label{sec:introduction}
The main purpose of the paper is the analysis of the decaying and scattering properties of the solution to the defocusing nonlinear fourth-order Hartree-Fock-Choquard equations (HFC4) in dimension $d \geq 3 $:

\begin{equation}\label{eq:HFC4}
\begin{cases}
i\partial_t u_j + \Delta^2 u_j -\sigma_1 \Delta u_j  +\sigma_2 V(x)  u_j +
\sum_{k=1}^N\mathcal F(u_j , u_k)=0,  \\ 
(u_j(0,\cdot))_{j=1}^N= (u_{j,0})_{j=1}^N \in H^2(\R^d)^N,
\end{cases}
\end{equation}
with $\sigma_1,\sigma_2=0,1$. We shall assume that
$V=V(x)$ is a nonnegative Schwartz  radial function such that 
\begin{equation}\label{eq.hopot0}
x \cdot \nabla V \leq 0
\end{equation}
and
\begin{equation}\label{eq.hopot}
\|V\|_{L^{\frac{d}{4}}}+\sup _{y \in \mathbb{R}^{d}} \int_{\mathbb{R}^{d}} \frac{V(x)}{|x-y|^{d-4}} d x < +\infty.
\end{equation}
The nonlinear term is given in the following general form
\begin{align}\label{eq.nonlinHF}
\mathcal F(u_j, u_k)
\\
=b_{jk} \squad{|x|^{-(d-\gamma_1)}*(|x|^{-\rho_{1}}| u_k|^p)}|x|^{-\rho_{1}} | u_j|^{p-2}  u_j
\nonumber\\
+b\round {\squad{|x|^{-(d-\gamma_2)}*|x|^{-\rho_{2}}| u_k|^2  } |x|^{-\rho_{2}}u_j-\squad{|x|^{-(d-\gamma_2)}* |x|^{-\rho_{2}}\overline u_k u_j  } |x|^{-\rho_{2}} u_k}.
\nonumber
\end{align}
Here, for all $j,k=1,\dots,N$, $u_j=u_j(t,x):\R\times\R^d\to\C$, $(u_j)_{j=1}^N=(u_1,\dots, u_N)$ and  $b, b_{jk}\geq0$
 are coupling constants. We require that the nonlinearity parameters $p$, $\gamma_{\kappa}$ and $\rho_{\kappa}$, $\kappa=1,2$, are $L^2$-$H^2$-intercritical, that is, when they satisfy the following conditions 
\begin{align}
\label{eq:bs}
\max (0, d-8)<\gamma_{\kappa}<d,\ \
 2\leq p<p^*(d), \ \  p^*(d)=
\begin{cases}
\infty \,  & \text{if} \ \  d=3,4, \\
\frac {d+\gamma_{1}+\rho_{1}}{d-4}  \,    &\text{if} \ \ d\geq5,
\end{cases}
\\
p> p_{1*}(d), \ \  p_{\kappa*}(d)=\frac{d+ \gamma_{\kappa}+4+\rho_{\kappa}}d,
\label{eq:bsII}
\\
0\leq  \rho_{\kappa}<\min\{n+\gamma_{\kappa},4(1+\gamma_{k}/d),8+\gamma_{\kappa}-d\},
\label{eq:bsIII}
\\
2\gamma_{\kappa}-4\rho_{\kappa}+d>0\quad \text{if} \quad 3\leq d\leq 4.
\label{eq:bsn}
\end{align}
We recall two important quantities linked to \eqref{eq:HFC4}. The mass
\begin{align}\label{eq.mass}
M(u_{j})(t)=\int_{\R^d}|u_{j}(t)|^2\,dx
\end{align}
and the energy  
\begin{equation}\label{eq:Ener}
\begin{split}
  E(u_{1},\dots,u_{N}) =  \sum_{j=1}^N\int_{\R^d}  \abs{\Delta{u_j}}^2\,dx+ \sigma_{1}\sum_{j=1}^N\int_{\R^d}  \abs{\nabla{u_j}}^2\,dx+ \sigma_{2} \sum_{j=1}^N\int_{\R^d} V(x) \abs{u_j}^2\,dx
  \\
 + \frac{1}{2p}\sum_{j,k=1 }^Nb_{jk} \int_{\R^d}\squad{|x|^{-(d-\gamma_1)}\ast
(|x|^{-\rho_{1}}|u_k|^{p})}|x|^{-\rho_{1}}|u_j|^{p} dx
\\
+\frac {b}2\sum_{j,k=1 }^N \int_{\R^d}\squad{|x|^{-(d-\gamma_2)}*( |x|^{-\rho_{2}}| u_k|^2  )} |x|^{-\rho_{2}}|u_j|^2\,dx
\\
-\frac {b}2\sum_{j,k=1 }^N \int_{\R^d} \squad{|x|^{-(d-\gamma_2)}*( |x|^{-\rho_{2}}\overline u_k u_j ) } |x|^{-\rho_{2}}u_k \overline u_j\, dx.
\end{split}
   \end{equation}
The Hartree-Fock system of $N$-particles is of fundamental interest from a physical point of view and plays an important role in several models of the mathematical physics. In fact, in the quantum mechanics it portrays the mean-field limit of large systems of bosons, \emph{i.e.} the so-called Bose-Einstein condensates, in view of the self-interactions of the charged particles. This scenario is considered, for example, in \cite{ES}, \cite{Len}, \cite{LR} and references therein. The Hartree-Fock equation was applied in \cite{Fo} in order to depict systems of charged fermions as well as an exchange term resulting from Pauli's principle and in \cite{Lip} to describe the fermions as an approximation of the equation, disregarding the effect of their fermionic nature. Furthermore, the Hartree-Fock-Choquard system, extensively treated in \cite{TarVenk}, ensured various other applications: in \cite{FrLe} for developing models of white dwarfs, for sketching an electron trapped in its own hole, as exhibited in \cite{ChS}, \cite{CSV} and in \cite{Pen}, for representing self-gravitating matter together with quantum entanglement and quantum information effects. On the other hand, the fourth-order Schr\"odinger equations were introduced in \cite{FIP} to portray dispersion in the propagation of intense laser beams in a medium with Kerr nonlinearity and successively employed in the theory of motion of a vortex filament in an incompressible fluid \cite{Kar}, \cite{KaS}, \cite{Se03} (see also \cite{HJ05} and \cite{HJ07}). It is natural, as a next step, to generalize the model to a higher order one, in our framework the HFC4 system \eqref{eq:HFC4}, given in a rather general way, which discusses various equations of motions. In addition, the second issue is to extend the scope of HFC4 to the short range-potentials. The relevant papers \cite{BJPSS} and \cite{BSS2} (see also the references inside) suggest also that it is necessary to bound the solutions in some $L^{2}$-based spaces, so to show that the evolution initial mixed states remains close to a (mixed) quasi-free state, evolving according to Hartree-Fock type equations, such as HFC4 systems. Influenced by this and by  \cite{CM}, \cite{CLM}, our main target is the analysis of the decay of the solutions to  \eqref{eq:HFC4} in the energy space, which means the decay of $L^r$-norms of the solutions, provided that $2<r\leq2d/(d-4)$, for $d\geq 5$, $2<r<\infty$ for $d=3,4$ and large-data in $H^2(\R^d)^N$. Namely, by carrying on the ideas initiated in \cite{Ta} and \cite{TarVenk} for systems of Schr\"odinger equations with local and non-local nonlinearities, we set up the Morawetz action, its tensor counterpart and get new bilinear Morawetz inequalities for \eqref{eq:HFC4}. A localized version of such inequalities and a contradiction argument imply the decay of $L^r$-norms of $(u_j(t,x))_{j=1}^N$, also when the HFC4 equation is lacking of translation invariance. The decay of the solution is a strong property of the HFC4 and quickly bears to the asymptotic completeness if one uses an appropriate reformulation of the theory in \cite{Ca}. We move now on the breakthroughs available in our paper. We underline that our results are new in the whole literature, not only for the HFC4 equations but also for the fourth-order Hartree-Fock equation ($b_{jk}=0$, for all $j,k=1, \dots, N$ in \eqref{eq:HFC4}).  
For what concerns the fourth-order Schr\"odinger-Choquard equation
\begin{equation}\label{eq:GenSchoq}
  \begin{cases}
i \partial_t u + (\Delta^{2} -\sigma_{1}\Delta+\sigma_{2}V)u =\lambda ({|x|^{-(d-\gamma_1)}}*
|x|^{-\rho_{1}}|u|^{p})|x|^{-\rho_{1}}|u|^{p - 2} u,\\
u(0,x)=u(x),
  \end{cases}
\end{equation}
with $\lambda<0$, that is \eqref{eq:HFC4}  when $N=1$, $\sigma_{1}=0$ and $\sigma_{2}=0$, we improve to the non-radial setting the various results available in \cite{Saa1}, \cite{Saa2}, including also $\sigma_{1}=1$, $d\geq 3$ and $\sigma_{2}>0$. Finally, we also enhance to the non-radial framework the scattering results in \cite{FWY}, because, as underlined in \cite{TarVenk}, our technique enables us to treat local and non-local terms in an unified way (see \cite{CT}). 
\\
\\
 Now, we state the first main result of this paper. Namely,

\begin{theorem}\label{thm:desl}
   Let be $2\leq p<p^*(d)$ and assume \eqref{eq.hopot0}, \eqref{eq.hopot}, \eqref{eq:bs}, \eqref{eq:bsIII}, \eqref{eq:bsn} being satisfied. Then for any $(u_{j,0})_{j=1}^N \in H^2(\R^d)^N,$ the unique global solution to \eqref{eq:HFC4}
  $(u_j(t,x))_{j=1}^N \in\mathcal C(\R, H^2(\R^d)^N)$, enjoys:
   \begin{align}\label{eq:desol0}
\lim_{t\rightarrow \pm \infty} \|u_j(t, \cdot)\|_{L^r(\R^d)}=0,
\end{align}
for all $j=1,\dots,N$, with $2<r<2d/(d-4),$ if $d\geq 5$ and $2<r<+\infty$ if $d=3, 4$, in the following cases
    \begin{itemize}
    \item[1)] if $d\geq 3$, $(b_{jj},\sigma_1)\neq(0,0),$ for all $j=1, \dots, N$, $(\rho_{1},\rho_{2})=(0,0)$ and $\sigma_2=0$;
    \item[2)] if $d\geq 5$, for  $(\rho_{1},\rho_{2}, \sigma_2)\neq(0,0,0)$.
    \end{itemize}
   \end{theorem}
   
 The direct consequence of this theorem is the scattering in the energy space.

\begin{theorem}\label{thm:scattHFC4}
   Let be $2\leq p<p^*(d)$, $p_{2*}(d)\leq 2$ and assume  \eqref{eq.hopot0}, \eqref{eq.hopot}, \eqref{eq:bs},  \eqref{eq:bsII}, \eqref{eq:bsIII},  \eqref{eq:bsn} being satisfied. Then for any $(u_{j,0})_{j=1}^N \in H^2(\R^d)^N,$ there exist $(u_0^{\pm})_{j=1}^N \in H^2(\R^d)^N$ such that the solution to \eqref{eq:HFC4}
  $(u_j(t,x))_{j=1}^N \in\mathcal C(\R, H^2(\R^d)^N)$, fulfills
    \begin{equation}\label{eq:scattering0}
      \lim_{t\to\pm\infty}\left\|u_j(t,\cdot)-e^{it(\Delta^2_x-\sigma_{1} \Delta_x+\sigma_{2}V)}u_{j,0}^{\pm}(\cdot)\right\|_{H^2(\R^d)}=0,
    \end{equation}
for any $j=1, \dots, N$, in the following cases
    \begin{itemize}
    \item[1)] if $d\geq 3$, $(b_{jj},\sigma_1)\neq(0,0),$ for all $j=1, \dots, N$, $p>2$, $(b, \rho_{1},\rho_{2})=(0, 0,0)$ and $\sigma_2=0$;
    \item[2)]  if $d\geq 5$, for $p=2$ or $(\rho_{1},\rho_{2}, \sigma_{2})\neq(0, 0,0)$.
    \end{itemize}
   \end{theorem}

\begin{remark}[Decay]\label{decay}
The decay of the solutions to \eqref{eq:HFC4} in $L^r$ spaces requires less restriction on the ranges of the parameters involved in the equation, which is a consequence of the fact that we just need satisfied that $(u_j(t,x))_{j=1}^N \in C(\R, H^2(\R^d)^N)$, assured by the Proposition \ref{CoLa} below. As well, all the covered different cases can be summarized as follow.
\begin{itemize}
\item[] \emph{Case $1)$ Theorem \ref{thm:desl}}. It embraces  the fourth-order Hartree-Fock equation (with $\sigma_{1}=1$, $N\geq2$, $b_{jk}=0$), the fourth-order Schr\"odinger-Choquard system ($b_{jj}\neq 0$, $b=0$, $N\geq 1$), the more general fourth-order HFC4, this means that the full interaction between the two terms in \eqref{eq.nonlinHF} is permitted. 
\item[] \emph{Case $2)$ Theorem \ref{thm:desl}}. It comprehends the fourth-order Hartree-Fock  equation (with $\sigma_{1}\geq 0)$, the fourth-order  inhomogeneous Schr\"odinger-Choquard system ($N\geq 1$) and  the more general fourth-order inhomogeneous HFC4 equation ($N\geq2$), both in the unperturbed and perturbed by a potential regime as well as if
$( \rho_{1},\rho_{2})\neq(0,0)$. Here the space dimension is $d\geq 5$ because \eqref{eq:HFC4} is not translation invariant.
\end{itemize}
\end{remark}

\begin{remark}[Scattering]\label{scatter}
The scattering of the solutions to \eqref{eq:HFC4} in $H^{2}(\R^{d})$ needs more summability given by  the Strichartz norms. This reflect to have more constrains on the parameters associated to the equation. Thus, the cases involved are the following.
\begin{itemize}
\item[] \emph{Case $1)$ Theorem \ref{thm:scattHFC4}}. It investigates the fourth-order Schr\"odinger-Choquard system only. 
\item[] \emph{Case $2)$ Theorem \ref{thm:scattHFC4}}. It embraces all the other models itemized in the previous Remark \ref{decay}.
\end{itemize}
\end{remark}

Surveying the literature, which is not so wide according to our knowledge, we are unaware of alike results, with the exception of the aforementioned \cite{Saa1} and \cite{Saa2}. The author prove, in the former, scattering for the solution to \eqref{eq:HFC4}, with $\rho_{1}=0,$ by using the radial concentration-compactenss method of \cite{Guo}, while in the latter establishes the radial decay of the solution to \eqref{eq:HFC4}, with $\rho_{1}>0,$ and the corresponding scattering. Additionally, we look to the interaction Morawetz estimates, which were first displayed for the Hartree-fock equations in the recent works \cite{CGMZ} and \cite{TarVenk}. \\
The paper is organized as follows. After the preliminary Section \ref{Prl}, we afford both the standard and tensor  Morawetz inequalities and their localized analogues in Section \ref{TMoraw}. Sections \ref{DecSolu} and \ref{NFC4scat} are dedicated to the proofs of the Theorems \ref{thm:desl} (decay) and \ref{thm:scattHFC4} (scattering), respectively. The final Section  \ref{appendix} is the Appendix, where an equivalence result between bilinear and tensor Morawetz identities is acquired.

\section{Preliminaries}\label{Prl}
\subsection{Notations.} We denote by $L_x^r$ the Lebesgue space $L^r(\R^{d})$, with $r\geq 1$ and respectively by $W^{s,r}_x$ and  $H^s_x$ the  inhomogeneous Sobolev spaces $W^{s,r}(\R^{d})=(1-\Delta)^{-\frac s2}L^r(\R^{d})$ and $H^2(\R^{d})=(1-\Delta)^{-\frac s2}L^2(\R^{d})$, with $s\geq 1$ (for more details see \cite{Ad}). We indicate by $B^{d}_{x}(r)$, the $d$-dimensional ball centered at $x\in\R^{d}$. Moreover, we set $A^N=A \times \dots \times A$, $N$-times, for any general set $A$ and integer $N\geq 1$.
We also utilize the symbol $\mathcal D_x$ (resp. $\mathcal D_y$) to unfold the dependence w.r.t. $x$ (resp. $y$) variable of a general differential operator $\mathscr D$. We adopt in the sequel the following notation: for any two positive real numbers $a, b,$ we write $a\lesssim b$  (resp. $a\gtrsim b$) to denote $a\leq C b$ (resp. $Ca\geq b$), with $C>0.$
\subsection{Inequalities.}
One can recall from \cite{FWY}, \cite{Pa1}, \cite{PaSh}, \cite{PaX}, the following

\begin{definition}\label{Sadm}
An exponent pair $(q, r)$ is biharmonic-admissible, in short $(q, r) \in \mathcal {B}$, if $2\leq q,r\leq \infty,$  $(q, r, d ) \neq (2,\infty, 4),$ and
\begin{align}\label{StrBharmP}
\frac 4q +\frac d {r}=\frac d2.
\end{align}
We say  that $(\tilde q', \tilde r')\in  \mathcal {B}'$  if it is the H\"older conjugate of the pair $(q, r)\in \mathcal {B}.$
\end{definition} 
Aggiungere motivazione
\begin{proposition}\label{Stri}
Let be two biharmonic-admissible pairs $(q,r)$, $(\widetilde q,\widetilde r)$ and  $V$ as in \eqref{eq.hopot}. Then we have, for $s=0,1$, that the following estimates 
\begin{align}
\label{eq:StriBiharm}
 \| \Delta_{x}^s e^{ it(\Delta^2_x-\sigma_{1} \Delta_x+\sigma_{2}V)} f\|_{L^q_t L^r_x} + \left \|\Delta_{x}^s\int_0^t e^{i (t-\tau) (\Delta^2_x-\sigma_{1} \Delta_x+\sigma_{2}V)} F(\tau)
d\tau\right\|_{L^q_t L^r_x}&\\\nonumber
\leq C\big (\|\Delta_{x}^s f\|_{ L^2_x}+\|\Delta_{x}^s F\|_{L^{\widetilde q'}_t
L^{\widetilde r'}_x}\big ),&
 \end{align}
 is fulfilled in the following cases:
 \begin{itemize}
 \item[a)] for $d\geq 3$, $\sigma_1 \geq 0$ and $\sigma_2=0$;
 \item[b)] for $d\geq 5$, $\sigma_{1}=0$ and $\sigma_2\neq 0.$
 \end{itemize}
 We retrieve also, for $d\geq 5$, the estimates
  \begin{align}\label{eq.casStrex}
  \|\Delta_{x}  e^{ it(\Delta^2_x+\sigma_{2} V)} f\|_{L_{t}^{q}L_{x}^{r}} + \left \| \Delta_{x}\int_0^t e^{ i (t-\tau) (\Delta^2_x+\sigma_{2} V)} F(\tau)
d\tau\right\|_{L^q_t L^r_x}&\\\nonumber\leq C\left(\left\|\Delta_{x} f\right\|_{L^{2}}+\|\nabla_{x} F\|_{L_{t}^{2}L_{x}^{\frac{2 d}{d+2}}}+\|\nabla_{x} F\|_{L_{t}^{2}L_{x}^{\frac{2 d}{d+4}}}\right)
\end{align}
and
\begin{align}\label{eq.casStrexDu}
\left\|(\Delta^2_x+ V)^{\frac{1}{2}} \int_{\mathbb{R}} e^{-i \tau (\Delta^2_x+ V)} F(\tau) d s\right\|_{L_{x}^{2}} \leq C\|\nabla_{x} F\|_{L_{t}^{2}L_{x}^{\frac{2 d}{d+2}}}.
\end{align}
 \end{proposition}
 It is also useful the following (see \cite{FWY})
\begin{proposition}\label{NormEquivF}
Let be $V$ as in \eqref{eq.hopot}. For any $1<r<\infty$ and $0<s^{*}<2$, one gets
 \begin{align}\label{eq.SobEquiv}
\left\|(\Delta^2_x)^{\frac{s^{*}}{2}} f\right\|_{L^{r}_{x}}\lesssim \left\|(\Delta^2_x+ V)^{\frac{s^{*}}{2}} f\right\|_{L^{r}_{x}} \lesssim\left\|(\Delta^2_x)^{\frac{s^{*}}{2}} f\right\|_{L^{r}_{x}}.
\end{align}
 \end{proposition}

\subsection{Well-posedness in energy space.} We summarize some of the well-posedness results for \eqref{eq:HFC4}, already appeared in \cite{CLM}, \cite{FWY}, \cite{Saa1}, \cite{Saa2}, which can be obtained by standard energy method (see Theorem 3.3.9 and Remark 3.3.12 in \cite{Ca}) and remain valid for the HFC4 systems. This is done in the following

\begin{proposition}\label{CoLa}
Assume  \eqref{eq.nonlinHF} is such as in Theorem \eqref{thm:desl}. Then for $u_{j,0}\in H^2_x$, with $j=1,\dots,N$, there exists a unique global solution $(u_j)_{j=1}^N \in  C(\R,H^2(\R^{d})^{N}) $ to \eqref{eq:HFC4}, moreover 
\begin{align}
  &M(u_j)(t)=\norma{u_j(0)}_{L^2_x} 
   \label{eq:mcon}
  \end{align}
 and 
  \begin{align}
  &E(u_1(t),\dots,u_N(t))=E(u_1(0),\dots,u_N(0))
  \label{eq:econs},
\end{align}
with $E(u_1(t),\dots,u_N(t))$ as in \eqref{eq:Ener}.
\end{proposition}

\section{Morawetz and Tensor Morawetz identities}\label{TMoraw}
The main aim of this section is to provide the Morawetz-type identities and inequalities. From now on we hide the $t$-variable for easiness, spreading it out only when it is required. We have
\begin{proposition}[Morawetz]\label{lem:mor}
Let $(u_j(t,x))_{j=1}^N \in C(\R, H^2(\R^d)^{N})$ be a global solution to the system \eqref{eq:HFC4},
 let $a=a(x): \R^d \rightarrow \R$ be a sufficiently regular function and introduce the $j$-action given, for any $j=1,\dots, N$, by
\begin{equation}\label{eq.singMor}
M_j(t)
=2 \Im\int_{\R^d} \overline{u_j}(x)\nabla u_j(x) \cdot\nabla a(x)\,dx.
\end{equation}
The following identity holds:
\begin{align}\label{eq:mor2}
 \sum _{j=1}^N\dot M_j(t)
=  \sum _{j=1}^N\int_{\R^d} (-\Delta^3a(x)+\sigma_1 \Delta^2a(x))|u_j(x)|^2+2\Delta^2a(x)|\nabla u_j(x)|^2\,dx&
\\
+
4  \sum _{j=1}^N\Re\int_{\R^d}\nabla u_j(x)D^2(\Delta-\sigma_1)a(x)\cdot\nabla \overline u_j(x)\,dx
\nonumber
\\
-8 \sum _{j=1}^N \Re\int_{\R^d}D^2 u_j(x)D^2a(x)D^2 \overline u_j(x)\,dx+\sigma_2 \sum _{j=1}^N\int_{\R^d}\nabla a(x) \cdot \nabla V(x)|u_j(x)|^2\,dx
\nonumber
\\
-\frac{2(p-2)}{p}\sum _{j,k=1}^Nb_{jk}
\int_{\R^d} \Delta a(x)\squad{|x|^{-\rho_{1}}\round{|x|^{-(d-\gamma_1)}*(|x|^{-\rho_{1}}|u_k|^p)}} | u_j(x)|^{p}\,dx
\nonumber\\
+
\frac 4 p\sum_{j,k=1}^N b_{jk}\int_{\R^d}\nabla a(x)\cdot \nabla \squad{|x|^{-\rho_{1}}\round{|x|^{-(d-\gamma_1)}*(|x|^{-\rho_{1}}| u_k|^p)}} | u_j(x)|^{p}\,dx
\nonumber\\
 -2b\sum_{j,k=1}^N\int_{\R^d}\Delta a(x)\squad{|x|^{-(d-\gamma_2)}*(|x|^{-\rho_2}| u_k|^2) }|x|^{-\rho_2}|u_j(x)|^2\,dx
\nonumber \\
  +2b\sum_{j,k=1}^N\int_{\R^d}\Delta a(x)\squad{|x|^{-(d-\gamma_2)}*(|x|^{-\rho_2}\overline u_k u_j ) }|x|^{-\rho_2} u_k(x)  \overline u_j(x)\,dx
 \nonumber\\
+2b\sum_{j,k=1}^N\int_{\R^d}\nabla a(x)\cdot\nabla\squad{|x|^{-\rho_{2}}\round{|x|^{-(d-\gamma_2)}*(-|x|^{\rho_{2}}| u_k|^2 )}} |u_j(x)|^2\, dx
\nonumber
 \\
 -2b\sum_{j,k=1}^N\int_{\R^d}\nabla a(x)\cdot \nabla \squad{|x|^{-\rho_{2}}\round{|x|^{-(d-\gamma_2)}*({|x|^{-\rho_{2}}\overline u_k u_j  }}} u_k(x)  \overline u_j(x) \, dx,
\nonumber
\end{align}
with $\sigma_{1}, \sigma_{2}=0,1$, $D^2a\in\mathcal M_{d\times d}(\R^d)$ is the hessian matrix of $a$, $\Delta^2a=\Delta(\Delta a)$
and $\Delta^3a=\Delta(\Delta^2a)$.
\end{proposition}
\begin{proof} We can assume without loosing generality that the solution of  \eqref{eq:HFC4} is smooth and decaying, switching to the general case by using in a final density argument in the space $ C(\R,H^2(\R^d)^N).$ Then, an integration by parts implies

  \begin{align}\label{eq.smor1}
 \sum _{j=1}^N \partial_t M_j(t)  =-2 \sum _{j=1}^N\Im  \int_{\R^d} \partial_t u_j(x) [\Delta a(x)\overline u_j(x)+2\nabla a(x)\cdot \nabla \overline u_j(x)]\,dx& \\
    =2 \sum _{j=1}^N\Re \int_{\R^d} i \partial_t u_j(x) [\Delta a(x)\overline u_j(x)+2\nabla a(x)\cdot \nabla \overline u_j(x)]\,dx& \nonumber
    \\ 
    =2 \sum _{j=1}^N\Re \int_{\R^d} [-\Delta^2 u_j(x) +\sigma_1\Delta u_j(x)
    ][(\Delta a(x)\overline u_j(x)+2\nabla a(x)\cdot \nabla \overline u_j(x)]\,dx&
 \label{eq.smor1a}
 \\
 -2\sigma_2 \sum _{j=1}^N\Re \int_{\R^d} V(x)u_j(x)
    [(\Delta a(x)\overline u_j(x)+2\nabla a(x)\cdot \nabla \overline u_j(x)]\,dx&
   \nonumber
    \\
   -2\sum_{j,k=1}^Nb_{jk}\int_{\R^d} \Delta a(x)\squad{|x|^{-(d-\gamma_1)}*(|x|^{-\rho_1}| u_j|^p)}|x|^{-\rho_1} | u_k(x)|^{p}\,dx
 \label{eq.smor1b}
\\
-4\sum_{j,k=1}^Nb_{jk}\Re\int_{\R^d} \nabla a(x)\squad{|x|^{-(d-\gamma_1)}*(|x|^{-\rho_1} |u_k|^p)}|x|^{-\rho_1} | u_j(x)|^{p-2}u_j(x)\nabla \overline u_j(x)\,dx
 \label{eq.smor1c}
    \\
 -2b\sum_{j,k=1}^N\int_{\R^d}\Delta a(x)\squad{|x|^{-(d-\gamma_2)}*(|x|^{-\rho_2}| u_k|^2) }|x|^{-\rho_2}|u_j(x)|^2\,dx
 \label{eq.smor1d}
 \\
  +2b\sum_{j,k=1}^N\int_{\R^d}\Delta a(x)\squad{|x|^{-(d-\gamma_2)}*(|x|^{-\rho_2}\overline u_k u_j)  }|x|^{-\rho_2} u_k(x)  \overline u_j(x)\,dx
 \nonumber\\
 -4b\sum_{j,k=1}^N\Re\int_{\R^d}\nabla a(x)\squad{|x|^{-(d-\gamma_2)}*(|x|^{-\rho_2}| u_k|^2) }|x|^{-\rho_2}u_j(x)\nabla\overline u_j(x)\, dx
 \label{eq.smor1e}
 \\
 +4b\sum_{j,k=1}^N\Re\int_{\R^d}\nabla a(x)\squad{|x|^{-(d-\gamma_2)}*(|x|^{-\rho_2}\overline u_k u_j)  }|x|^{-\rho_2} u_k(x) \nabla \overline u_j(x) \, dx.
\nonumber
\end{align}
  By proceeding as in \cite{Ta} in combination with a further integrations by parts of the term involving $V(x)$, we get 
    \begin{align}\label{eq.smor2}
 \eqref{eq.smor1a}=  \sum _{j=1}^N \int_{\R^d} (-\Delta^3a(x)+\sigma_1 \Delta^2a(x))|u_j(x)|^2+2\Delta^2a(x)|\nabla u_j(x)|^2\,dx&
\\
+
4 \sum _{j=1}^N \int_{\R^d}\nabla u_j(x)D^2(\Delta-\sigma_1)a(x)\cdot\nabla \overline u_j(x)\,dx
-8 \sum _{j=1}^N\int_{\R^d}D^2 u_j(x)D^2a(x)D^2 \overline u_j(x)\,dx.
\nonumber
\\
+\int_{\R^d}\nabla a(x) \cdot \nabla V(x)|u_j(x)|^2\,dx.
    \nonumber
  \end{align}
  By arguing as in \cite{TarVenk} we have 
  \begin{align}\label{eq.smor3}
 \eqref{eq.smor1b}+\eqref{eq.smor1c}
 \\
 = -2\sum_{j,k=1}^Nb_{jk}\int_{\R^d} \Delta a(x)\squad{|x|^{-(d-\gamma)}*(|x|^{-\rho_1}| u_j|^p)} |x|^{-\rho_1}| u_j(x)|^{p}\,dx
\nonumber\\
-4\sum_{j,k=1}^Nb_{jk}\int_{\R^d} \nabla a(x)\squad{|x|^{-(d-\gamma)}*(|x|^{-\rho_2}| u_k|^p)} |x|^{-\rho_1}| u_j(x)|^{p-2}\frac 12\nabla |u_j(x)|^2\,dx
    \nonumber
    \\
    =-\frac{2(p-2)}{p}\sum _{k=1}^Nb_{jk}
\int_{\R^d} \Delta a(x)\squad{|x|^{-(d-\gamma_1)}*(|x|^{-\rho_1}|u_k|^p)} |x|^{-\rho_1}| u_j(x)|^{p}\,dx 
\nonumber
    \\
    +\frac 4p\sum_{j,k=1}^Nb_{jk}\int_{\R^d} \nabla a(x)\nabla\squad{|x|^{-\rho_1}\round{ |x|^{-(d-\gamma)}*(|x|^{-\rho_1}| u_k|^p)}} | u_j(x)|^{p}\,dx
    \nonumber
    \end{align}
and
   \begin{align}\label{eq.smor4}
\eqref{eq.smor1e}
 \\
-4b\sum_{j,k=1}^N\Re\int_{\R^d}\nabla a(x)\squad{|x|^{-(d-\gamma_2)}*(|x|^{-\rho_2}| u_k|^2 )}|x|^{-\rho_2}u_j(x)\nabla\overline u_j(x)\, dx
\nonumber
 \\
 +4b\sum_{j,k=1}^N\Re\int_{\R^d}\nabla a(x)\squad{|x|^{-(d-\gamma_2)}*(|x|^{-\rho_2}\overline u_k u_j)  }|x|^{-\rho_2} u_k(x) \nabla \overline u_j(x) \, dx
\nonumber
\\
=-2b\sum_{j,k=1}^N\Re\int_{\R^d}\nabla a(x)\squad{|x|^{-(d-\gamma_2)}*(|x|^{-\rho_2}| u_k|^2 )}|x|^{-\rho_2}\nabla |u_j(x)|^2
\nonumber
 \\
 +2b\sum_{j,k=1}^N2\Re\int_{\R^d}|x|^{-\rho_2}|y|^{-\rho_2}\nabla a(x)\frac{\overline u_k(y)   u_k(x) \nabla u_j(y)  \overline u_j(x) }{|x-y|^{d-\gamma_2}}   \, dxdy=
\nonumber
\\
=-\eqref{eq.smor1d}+2b\sum_{j,k=1}^N\int_{\R^d}\nabla a(x)\cdot \nabla \squad{|x|^{-\rho_2}\round{|x|^{-(d-\gamma_2)}*(|x|^{-\rho_2}| u_k|^2 )}} |u_j(x)|^2\, dx
\nonumber
 \\
 -2b\sum_{j,k=1}^N\int_{\R^d}\nabla a(x)\cdot \nabla \squad{|x|^{-\rho_2}\round{|x|^{-(d-\gamma_2)}*(|x|^{-\rho_2}\overline u_k u_j)} } |x|^{-\rho_2} u_k(x)  \overline u_j(x) \, dx.
\nonumber
 \end{align}
Finally, by taking into account \eqref{eq.smor2}, \eqref{eq.smor3}, \eqref{eq.smor4}, the identity \eqref{eq.smor1} gives rise to \eqref{eq:mor2}.
  \end{proof}

By an application of the above lemma, we can now go over to the proof of the tensor Morawetz identities. More precisely, we have

\begin{lemma}[Tensor Morawetz]\label{lem:tenmor}
Assume $d\geq 3$ and let $(u_j(t,x))_{j=1}^N \in C(\R, H^2(\R^d)^N)$ be a global solution to system \eqref{eq:HFC4}. Further, let us denote by $z_{j,\ell}(x,y)=u_j(x)u_{\ell}(y)$, $a(x,y)=|x-y|$ and set the $j,\ell$-tensor action
\begin{equation}\label{eq:tenmor1}
  \mathcal M_{j,\ell}(t) =2\Im\int_{\R^{2d}}\overline{z_{j,\ell}}(x,y)(\nabla_{x},\nabla_{y})z_{j,\ell}(x,y)  \cdot(\nabla_{x},\nabla_{y})a(x,y)\,dxdy.
\end{equation}
Then the following inequality holds:
  \begin{align}\label{eq:tenmor2}
  \sum_{j,\ell=1}^N\mathcal{ \dot{M}}_{j,\ell}(t)
  \\
  \lesssim 2\sum_{j,\ell=1}^N\int_{\R^{2d}}
  \round{\Xi(x,y)+\sigma_{1} \Delta^2_xa(x,y)| u_j(x)|^2|u_{\ell}(y)|^2} \,dxdy
 \nonumber
 \\
- \sum_{j,\ell, k}^N \widetilde b_{jk}\int_{\R^{2d}}\Delta_x a(x,y) \squad{|x|^{-(d-\gamma_1)}* (|x|^{-\rho_{1}}| u_k(x)|^p)}|x|^{-\rho_{1}}|u_j(x)|^{p}|u_{\ell}(y)|^2\,dxdy
\nonumber
 \\
+ \mathcal R(t),
\nonumber
\end{align}
with $ \widetilde b_{jk}=4b_{jk}(p-2)/p,$
\begin{equation}\label{eq.leadtherm}
\Xi(x,y)=
\begin{cases}
 \Delta^2_xa(x,y) \nabla_x| u_j(x)|^2\nabla_y|u_{\ell}(y)|^2 ,& \text{if} \ \ \ d\geq3,\cr
\\
- \Delta^3_xa(x,y) | u_j(x)|^2|u_{\ell}(y)|^2, & \text{if} \ \ \ d\geq5
\end{cases}
\end{equation}
and 
  \begin{align}\label{eq:tenmor3}
 \mathcal R(t)
  \\
 =2\sigma_2\sum_{j,\ell=1}^N \int_{\R^{2d}}\nabla_{x} a(x,y) \cdot \nabla V(x)| u_j(x)|^2|u_{\ell}(y)|^2\,dxdy&
 \nonumber
 \\
+ \sum_{j, k, \ell}^Nb_{jk}\int_{\R^{2d}}\nabla_x a(x,y) \cdot \nabla_x |x|^{-\rho_{1}} \squad{|x|^{-(d-\gamma_1)}* (|x|^{-\rho_{1}}| u_k(x)|^p)}|u_j(x)|^{p}|u_{\ell}(y)|^2\,dxdy
\nonumber
\nonumber\\
+2b\sum_{j,k, \ell=1}^N\int_{\R^{2d}}\nabla a(x,y)\cdot\nabla |x|^{-\rho_{2}}\squad{|x|^{-(d-\gamma_2)}*(|x|^{-\rho_{2}}| u_k|^2 )} |u_j(x)|^2|u_{\ell}(y)|^2\,dxdy\nonumber
 \\
 -2b\sum_{j,k, \ell=1}^N\int_{\R^{2d}}\nabla a(x,y)\cdot \nabla |x|^{-\rho_{2}}\squad{|x|^{-(d-\gamma_2)}*({|x|^{-\rho_{2}}\overline u_k u_j  )}}u_k(x)  \overline u_j(x) |u_{\ell}(y)|^2\,dxdy.
\nonumber
\end{align}
\end{lemma}

\begin{proof}
As before, we prove the identities for a smooth, decaying solution to \eqref{eq:HFC4}.
 First, we observe that,
 \begin{align}\label{eq.tenid}
  i \partial_t z_{j,\ell}(x,y)+ (\Delta_{x,y}^2-\sigma\Delta_{x,y})z_{j,\ell} (x,y)
=- (V(x)+V(y))u_{j}(x) u_{\ell} (y) \\
-\sum_{k=1}^N\mathcal F(u_j(x) , u_k(x))u_{\ell}(y)-\sum_{k=1}^N\mathcal F(u_{\ell}(y) , u_k(y))u_{j}(x),
  \nonumber
 \end{align}
with $\Delta^s_{x,y}=\Delta^s_{x}+\Delta^s_{y}$, for $s=1,2$. Then, we differentiate the tensor action w.r.t. time variable, achieving
  \begin{align}\label{eq.TenMor1}
  \mathcal {\dot M}_{j,\ell}(t)\\
  =2\Re\int_{\R^{2d}} i \partial_t z_{j,\ell}(x,y) [\Delta_{x,y} a(x,y)\overline z_{j,\ell}(x,y)+2(\nabla_{x},\nabla_{y}) a(x,y)\cdot (\nabla_{x},\nabla_{y}) \overline z_{j,\ell}(x,y)]\,dx&
    \nonumber\\ 
    =2\Re\int_{\R^{2d}} [\Delta_{x,y}^2 z_{j,\ell}(x,y) -\sigma_1\Delta_{x,y}  z_{j,\ell}(x,y)]
    [(\Delta_{x,y} a(x,y)\overline z_{j,\ell}(x,y)]\,dxdy& 
    \nonumber\\
   + 4\Re\int_{\R^{2d}}  [\Delta_{x,y}^2 z_{j,\ell}(x,y) -\sigma_1\Delta_{x,y}  z_{j,\ell}(x,y)]
     [(\nabla_{x},\nabla_{y}) a(x,y)\cdot (\nabla_{x},\nabla_{y}) \overline z_{j,\ell}(x,y)]\,dxdy& 
    \nonumber\\
    +\mathcal N^p_{j,\ell}(t)+\mathcal N_{j,\ell}(t)+ \mathcal N^V_{j,\ell}(t),
     \nonumber
    \end{align}
   where, by the identity \eqref{eq.tenid} and after exploiting the
  symmetry of $a(x,y)$ in combination with the Fubini's Theorem, we have
     \begin{align}\label{eq.imoranl}
\sum_{j, \ell=1}^N \mathcal   N^p_{j,\ell}(t)
 \\
=  -4\sum_{j, k, \ell=1}^Nb_{k}\int_{\R^{2d}}\left( \Delta_x a(x)\squad{|x|^{-(d-\gamma_{1})}*(|x|^{-\rho_{1}}| u_j|^p)} |x|^{-\rho_{1}}| u_k(x)|^{p}|u_{\ell}(y)|^2\right . 
\nonumber\\
\left . -2\Re\nabla_x a(x)\squad{|x|^{-(d-\gamma_1)}*(|x|^{-\rho_{1}}| u_k|^p)} \cdot |x|^{-\rho_{1}}| u_j(x)|^{p-2}u_j(x)\nabla_x \overline u_j(x)|u_{\ell}(y)|^2\,dxdy\right),
\nonumber
   \end{align}
      \begin{align}\label{eq.imorbnl}
\sum_{j, \ell=1}^N \mathcal   N_{j,\ell}(t)
 \\
 =   -4b\sum_{j, k, \ell=1}^N\Re\int_{\R^{2d}}\Delta_x a(x)\squad{|x|^{-(d-\gamma_2)}*(|x|^{-\rho_{2}}| u_k|^2 )}|x|^{-\rho_{2}}|u_j(x)|^2|u_{\ell}(y)|^2\,dxdy
\nonumber \\
  +4b\sum_{j, k, \ell=1}^N\Re\int_{\R^{2d}}\Delta_x a(x)\squad{|x|^{-(d-\gamma_2)}* (|x|^{-\rho_{2}}\overline u_k u_j)  }|x|^{-\rho_{2}} u_k(x)  \overline u_j(x)|u_{\ell}(y)|^2\,dxdy
  \nonumber
\\
 -8b\sum_{j, k, \ell=1}^N\Re\int_{\R^{2d}}\nabla_x a(x)\squad{|x|^{-(d-\gamma_2)}*(|x|^{-\rho_{2}}| u_k|^2) }|x|^{-\rho_{2}}u_j(x)\nabla\overline u_j(x)|u_{\ell}(y)|^2\,dxdy
\nonumber
 \\
 +8b\sum_{j, k, \ell=1}^N\Re\int_{\R^{2d}}\nabla_x a(x)\squad{|x|^{-(d-\gamma_2)}*(|x|^{-\rho_{2}}\overline u_k u_j)  }|x|^{-\rho_{2}} u_k(x) \nabla \overline u_j(x) |u_{\ell}(y)|^2\,dxdy
\nonumber   
\end{align}
and
      \begin{align}
\sum_ {j,\ell=1}\mathcal   N^V_{j,\ell}(t)
 \\
 =-\sigma_2\sum_{j, \ell=1}\Re \int_{\R^{2d}}V(x)u_{j}(x) 
    [(\Delta_{x} a(x,y)\overline u_{j}(x)+2\nabla_{x} a(x,y)\cdot \nabla_{x}\overline u_{j}(x)]|u_{\ell}(y)|^2\,dxdy.& 
    \nonumber
 \end{align}
\emph{Bilaplacian's terms.} We start by handling the terms in \eqref{eq.tenid} involving $\Delta^2_{x,y}$. We apply  \eqref{eq:mor2} (see \cite{MWZ} and \cite{Ta}), then the Fubini's Theorem allows us to write
  \begin{align}\label{eq:2iniz}
I_1(t)=  2\sum_{j,\ell=1}^N\Re \int_{\R^{2d}} \Delta_{x,y}^2 z_{j,\ell}(x,y) 
    [(\Delta_{x,y} a(x,y)\overline z_{j,\ell}(x,y)]\,dxdy& \\
    + 4 \sum_{j,\ell=1}^N\Re \int_{\R^{2d}} \Delta_{x,y}^2 z_{j,\ell}(x,y) 
    [(\nabla_{x},\nabla_{y}) a(x,y)\cdot (\nabla_{x},\nabla_{y}) \overline z_{j,\ell}(x,y)]\,dxdy& 
    \nonumber\\
=
  2\sum_{j,\ell=1}^N\int_{\R^{2d}}
 \Delta^2_xa(x,y)\round{\nabla_x|u_j(x)|^2\nabla_y|u_{\ell}(y)|^2+2|\nabla_x u_j(x)|^2|u_{\ell}(y)|^2}\,dxdy&
   \nonumber
  \\
   \,8 \sum_{j,\ell=1}^N\Re\int_{\R^{2d}}
\nabla_xu_{j}(x)D^2_x\Delta_xa(x,y)\nabla_x\overline u_{j}(x) |u_{\ell}(y)|^2\,dxdy &
 \nonumber\\
 -16 \sum_{j,\ell=1}^N\Re\int_{\R^{2d}}D_x^2 u_{j}(x)D_x^2a(x,y)D_x^2 \overline u_{j}(x) |u_{\ell}(y)|^2\,dxdy,&
\nonumber
   \end{align}
  where in the third line of the above \eqref{eq:2iniz} we applied the identity 
  \begin{equation}\label{eq:eqval}
    \begin{split}
    -\int_{\R^{2d}} \Delta_x^3a(x,y) | u_j(x)|^{2}|u_{\ell}(y)|^2\,dxdy=
    \\
    \int_{\R^{2d}}\nabla_x | u_j(x)|^{2}\cdot\nabla_y|u_{\ell}(y)|^2\Delta_x^2a(x,y)\,dxdy.
    \end{split}
  \end{equation}
  Let us exploit  the fact that $a(x,y)=|x-y|$. We achieve 
  \begin{equation}\label{eq.delta1}
\Delta_x|x-y|=
\begin{cases}
\frac{d-1}{|x-y|}    &\text{if} \ \ \ d\geq2,\cr
\\
c_{1}\delta_{x=y} &\text{if} \ \ \ d=1,
\end{cases}
\end{equation}
\begin{equation}\label{eq.delta2}
\Delta^2_x|x-y|=
\begin{cases}
-\frac{(d-1)(d-3)}{|x-y|^3}   &\text{if} \ \ \ d\geq4,\cr
\\
-c_{2}\delta_{x=y} &\text{if} \ \ \ d=3
\end{cases}
\end{equation}
and
\begin{equation}\label{eq.delta3}
\Delta^3_x|x-y|=
\begin{cases}
\frac{(d-3)(d-5)}{|x-y|^5}  &\text{if} \ \ \ d\geq6,\cr
\\
c_{3}\delta_{x=y}, &\text{if} \ \ \ d=5,
\end{cases}
\end{equation}
for some $c_{1}, c_{2},c_{3}>0$. Furthermore, if one sets $\nabla_{v}^{\bot}f=\nabla f- (v\cdot \nabla f)v/|v|^2$, for $v\in\R^{d}$, it fulfills
\begin{equation}\label{eq.quad1}
D^2_x u_j(x)D^2_x|x-y|D^2_x \overline u_j(x)\geq 
\frac{(d-1)}{|x-y|^3}|\nabla u(x)-\nabla_{x-y}^{\bot}u(x)|^2,
\end{equation}
for $d\geq2$ (we refer to \cite{LS}) and  
\begin{equation}\label{eq.quad2}
\nabla_x u(x)D_x^2\Delta^2_x|x-y|\nabla_x \overline u(x)
=-\frac{(d-1)}{|x-y|^3}(|\nabla_{x-y}^{\bot}u(x)|^2- 2|\nabla u(x)-\nabla_{x-y}^{\bot}u(x)|^2),
\end{equation}
(see \cite{MWZ} for more details). We have also, by the Fourier transform and Plancherel's identity, that
\begin{align}\label{eq.equiv2}
\sum_{j,\ell=1}^N\int_{\R^{2d}}
 \Delta^2_xa(x,y)\nabla_x|u_j(x)|^2\nabla_y|u_{\ell}(y)|^2\,dxdy&\\
 =-\round{\Delta_x\sum_{j,\ell=1}^N|u_j(x)|^2 ,(-\Delta_x)^{\frac{1-d}2} \Delta_x \sum_{j,\ell=1}^N|u_{\ell}(x)|^2}\leq 0,
  \nonumber
\end{align}
for any $d\geq 1$. Finally, By gathering \eqref{eq.quad1}, \eqref{eq.quad2}, \eqref{eq.equiv2} and \eqref{eq:2iniz} we attain $I_1(t)\leq 0.$\\
\emph{Laplacian's terms.} We will consider the terms involving $\Delta_{x,y}$. The approach of \cite{CGT}, \cite{Ta} and the Fubini's Theorem, bring to
  \begin{align}\label{eq.laplt}
 I_2(t)= -2\sum_{j,\ell=1}^N\sigma_1\Re \int_{\R^{2d}} \Delta_{x,y} z_{j,\ell}(x,y) 
    [(\Delta_{x,y} a(x,y)\overline z_{j,\ell}(x,y)]\,dxdy& \\
    - 4\sum_{j,\ell=1}^N\sigma_1\Re \int_{\R^{2d}} \Delta_{x,y} z_{j,\ell}(x,y) 
    [(\nabla_{x},\nabla_{y}) a(x,y)\cdot (\nabla_{x},\nabla_{y}) \overline z_{j,\ell}(x,y)]\,dxdy& 
    \nonumber\\
=
  2\sum_{j,\ell=1}^N\int_{\R^{2d}}
 \Delta^2_xa(x,y) | u_j(x)|^2|u_{\ell}(y)|^2\,dxdy&
\nonumber\\
    - 4\sigma_1\sum_{j,\ell=1}^N \Re \int_{\R^{2d}}  \nabla_{x}u_{j}(x)D^2_{xy}a(x,y)\nabla_{x}\overline u_{j}(x)|u_{\ell}(y)|^2\,dxdy&
    \label{eq.laplta}\\
 -4\sigma_1\sum_{j,\ell=1}^N\Re\int_{\R^{2d}} \nabla_{y}u_{\ell}(y)D^2_{xy}a(x,y)\nabla_{y}\overline u_{\ell}(y)|u_{\ell}(x)|^2\,dxdy&
 \label{eq.lapltb}\\
  +8\sigma_1\sum_{j,\ell=1}^N\int_{\R^{2d}}
  \Im(\overline u_{j}(x)\nabla_{x}u_{j}(x))D^2_{xy}a(x,y)
  \Im(\overline u_{\ell}(y)\nabla_{y}u_{\ell}(y))\,dxdy.&
  \label{eq.lapltc}
   \end{align}
One can verifies that (see Lemma \ref{eq.MorEqiv}, we refer also to \cite{TarVenk}) 
  \begin{align}\label{eq.lapl2}
  \eqref{eq.laplta}+\eqref{eq.lapltb}+\eqref{eq.lapltc}
  \\=-4\sigma_1\int_{\R^{2d}}
 \left(H_{jk} D^2_{x}\phi(|x-y|)\overline{H}_{j\ell}+G_{j\ell} D^2_{x}\phi(|x-y|)\overline{G}_{jk}\right)\,dxdy&,  
 \nonumber
\end{align}
where
   \begin{align*}
    H_{j\ell} %= C^{\mu\mu}_j(t,x,y)
    &
    :=u_{j}(t,x)\nabla_{y}\overline{u_{j}(t,y)}+\nabla_{x}u_{\ell}(t,x)\overline{u_{\ell}(t,y)},
    \\
    G_{j\ell}%= D^{\mu\mu}_j(t,x,y)
    &
    :=u_{j}(t,x)\nabla_{y}u_{\ell}(t,y)-\nabla_{x}u_{j}(t,x)u_{\ell}(t,y).
  \end{align*}
  Thus, since $a(x,y)$ is a convex function satisfying \eqref{eq.delta2}, one gets the chain of inequalities
  \begin{equation}\label{eq.laplt3}
  I_{2}(t)\leq  2\sum_{j,\ell=1}^N\int_{\R^{2d}}
 \Delta^2_xa(x,y) | u_j(x)|^2|u_{\ell}(y)|^2\,dxdy \leq 0.
 \end{equation}
\emph{Nonlinear terms.} We will start by treating the term $\mathcal   N^p_{j,\ell}(t)$. We have in fact
\begin{align}  \label{eq:nnlna}
 N^p_{j,\ell}(t)
\\
=\sum_{k=1}^N\left (-\widetilde b_{jk}\int_{\R^{2d}}\Delta_xa(x,y) \squad{|x|^{-(d-\gamma_1)}*(|x|^{-\rho_{1}}| u_k(x)|^p)}|x|^{-\rho_{1}}|u_j(x)|^{p}|u_{\ell}(y)|^2\,dxdy\right.&
 \nonumber\\
\left.+\frac {8b_{jk}}p\int_{\R^{2d}} \nabla_xa(x,y) \cdot \nabla_x \round{|x|^{-\rho_{1}}\squad{|x|^{-(d-\gamma_1)}*(|x|^{-\rho_{1}}| u_k|^p)}}|u_j(x)|^{p}|u_{\ell}(y)|^2\,dxdy\right),&
 \nonumber
\end{align}
where $ \widetilde b_{jk}$ is as in \eqref{eq:tenmor2}.
We notice also that, by means of
\begin{equation}\label{eq.dMor}
\nabla_xa(x,y)=\frac{x-y}{|x-y|},
\end{equation}
we can handle the last term in \eqref{eq:nnlna} as
\begin{align}  \label{eq:nonlin2}
\sum_{j,k,\ell=1}^N\left(b_{jk}^*\int_{\R^{2d}}
 \frac{(x-y)\cdot(x-z)|u_j(x)|^{p}| u_k(z)|^{p}}{|x-y||x-z|^{d-\gamma_1+2}}|x|^{-\rho_{1}}|y|^{-\rho_{1}} |u_{\ell}(y)|^2 \,dxdydz\right.
 \\
\left.+\frac {8b_{jk}}p\int_{\R^{2d}}
 \nabla_xa(x,y) \cdot \nabla_x |x|^{-\rho_{1}}\squad{|x|^{-(d-\gamma_1)}*(|x|^{-\rho_{1}}| u_k|^p)}|u_j(x)|^{p}|u_{\ell}(y)|^2\,dxdy\right)
 \nonumber\\
=\sum_{\substack{j,k, \ell=1}}^N\left(\frac 12 b_{jk}^*\int_{\R^{2d}}
 \frac{|x|^{-\rho_{1}}|y|^{-\rho_{1}}}{|x-z|^{d-\gamma_1+2}}|u_j(x)|^{p}| u_k(z)|^{p}K_{\ell}(x,z)\,dxdz\right.
 \nonumber\\
 \left.+\frac {8b_{jk}}p\int_{\R^{2d}}
 \nabla_xa(x,y) \cdot \nabla_x |x|^{-\rho_{1}}\squad{|x|^{-(d-\gamma_1)}*(|x|^{-\rho_{1}}| u_k|^p)}|u_j(x)|^{p}|u_{\ell}(y)|^2\,dxdy\right),
 \nonumber
\end{align}
with $b_{jk}^*=-8b_{jk}(d-\gamma_1)/p$ and where
\begin{equation}\label{eq.Kern}
 K_{\ell}(x,z)=(x-z)\cdot \int_{\R^d}|u_{\ell}(y)|^2\round{\frac{x-y}{|x-y|}-\frac{z-y}{|z-y|}}\, dy.
\end{equation}
Then, the inequality 
\begin{align}\label{eq.ineq}
(x-z) \cdot \round{\frac{x-y}{|x-y|}-\frac{z-y}{|z-y|}}
\\
=\round{|x-y||z-y|-(x-y)\cdot(z-y)}\round{\frac{|x-y|+|z-y|}{|x-y||z-y|}}\geq 0,
\nonumber
\end{align}
enhances to
\begin{align}\label{eq.inf}
\inf_{(x,y)\in\R^{2d} }K_{\ell}(x,z)\geq 0.
\end{align}
By gathering the above \eqref{eq.inf} with \eqref{eq:nonlin2} we obtain then, for any $t\in\R$,
\begin{align}  \label{eq:nnlna11}
\sum_{j,\ell=1}^N\mathcal N^p_{j,\ell}(t)
\\
\leq -\sum_{j,k,\ell=1}^N\left(\widetilde b_{jk}\int_{\R^{2d}}\Delta_xa(x,y) \squad{|x|^{-(d-\gamma_1)}*(|x|^{-\rho_{1}}| u_k(x)|^p)}|x|^{-\rho_{1}}|u_j(x)|^{p}|u_{\ell}(y)|^2\right.
\nonumber
\\
+\left.\frac {8 b_{jk}}{p} \nabla_xa(x,y)\cdot \nabla_{x} |x|^{-\rho_{1}} \squad{|x|^{-(d-\gamma_1)}*(|x|^{-\rho_{1}}| u_k(x)|^p)}|u_j(x)|^{p}|u_{\ell}(y)|^2\,dxdy\right).
\nonumber
\end{align}
 We consider now the term 
\begin{align} 
\sum_{j,\ell=1}^N\mathcal N_{j,\ell}(t)&
\\
=-2\beta(d-\gamma_2)\sum_{\ell=1}^N\int_{\R^{2d}}
 \frac{|x|^{-\rho_{2}}|y|^{-\rho_{2}}}{|x-z|^{d-\gamma_2+2}}\round{\varkappa(x,x) \varkappa(z,z)-|\varkappa(x,z)|^2 }K_{\ell}(x,z)\,dxdy
 \nonumber
 \\
 +2b\sum_{j,k,\ell=1}^N\int_{\R^{2d}}\nabla_{x} a(x)\cdot\nabla_{x}|x|^{-\rho_{2}}\squad{|x|^{-(d-\gamma_2)}*(|x|^{-\rho_{2}}| u_k|^2 )} |u_j(x)|^2|u_{\ell}(y)|^2\,dxdy
\nonumber
 \\
 -2b\sum_{j,k,\ell=1}^N\int_{\R^{2d}}\nabla_{x} a(x)\cdot \nabla_{x} |x|^{-\rho_{2}}\squad{|x|^{-(d-\gamma_2)}*({|x|^{-\rho_{2}}\overline u_k u_j  )}} u_k(x)  \overline u_j(x) |u_{\ell}(y)|^2\,dxdy,
\nonumber
\end{align}
with $K(x,z)$ as in \eqref{eq.Kern} and 
\begin{align}%\label{eq.HFnonli}
\varkappa(x,z)=\sum_{j=1}^N u_j(x)\overline u_j(z), \, \, \,  \, \, \,\varkappa(x)=\varkappa(x,x).
\end{align}
The Cauchy-Schwartz inequality bears to the bound $|\eta(x,z)|^2\leq \varkappa(x)\varkappa(z)$, for any $x,y\in\R^d$. This fact and  \eqref{eq.inf}, imply 
\begin{align}\label{eq.HFNonl}
\sum_{j,\ell=1}^N\mathcal N_{j,\ell}(t)&
 \\
 \leq 2b\sum_{j,k,\ell=1}^N\int_{\R^{2d}}\nabla_{x} a(x)\cdot\nabla_{x}|x|^{-\rho_{2}}\squad{|x|^{-(d-\gamma_2)}*(|x|^{-\rho_{2}}| u_k|^2 )} |u_j(x)|^2|u_{\ell}(y)|^2\,dxdy
\nonumber
 \\
 -2b\sum_{j,k,\ell=1}^N\int_{\R^{2d}}\nabla_{x} a(x)\cdot \nabla_{x} |x|^{-\rho_{2}}\squad{|x|^{-(d-\gamma_2)}*({|x|^{-\rho_{2}}\overline u_k u_j  )}} u_k(x)  \overline u_j(x) |u_{\ell}(y)|^2\,dxdy.
\nonumber
\end{align}

Finally we consider $N^V_{j,\ell}(t)$. An easy calculations yields 
  \begin{align}\label{eq.intmorpot}
\sum_{j,\ell=1}^N \mathcal   N^V_{j,\ell}(t)
 \\
 =2\sigma_2\sum_{j,\ell=1}^N \int_{\R^{2d}}\nabla_{x} a(x) \cdot \nabla_{x} V(x)|u_j(x)|^{2}|u_{\ell}(y)|^2\,dxdy.& 
    \nonumber
 \end{align}
By combining now \eqref{eq:eqval}, \eqref{eq.laplt3}, \eqref{eq:nnlna11}, \eqref{eq.HFNonl} and \eqref{eq.intmorpot}, we attain the inequality \eqref{eq:tenmor2}.
 \end{proof}
 We need also the following 
 
 \begin{corollary} \label{singmor}
Assume $d\geq 5$ and let $V: \mathbb{R}^{d} \rightarrow \mathbb{R}$ be a function satisfying \eqref{eq.hopot0}. If $(u_j(t,x))_{j=1}^N \in C(\R, H^2(\R^d)^N)$ be a global solution to system \eqref{eq:HFC4}  and $a(x,0)=a(x)$, then it holds that,
\begin{align}\label{eq.singmor}
\sum_{j=1}^N \dot M_j(t)\leq \sum_{j=1}^N \int_{\R^d} (-\Delta^3a(x)+\sigma_1 \Delta^2a(x))|u_j(x)|^{2}\,dx &\\
+ 2\sum_{j=1}^N  \int_{\mathbb{R}^{d}} \nabla a(x)\cdot \nabla V(x)|u_j(x)|^{2} d x&
\nonumber
\\
-\frac{2\beta_{1}}{(d-1)p}\sum _{j,k=1}^Nb_{jk}
\int_{\R^d}\Delta a(x)\squad{|x|^{-(d-\gamma_1)}*(|x|^{-\rho_{1}}|u_k|^p)} |x|^{-\rho_{1}} | u_j(x)|^{p}\,dx
\nonumber\\
 -\frac {2b\beta_{2}}{d-1}\sum_{j,k=1}^N\int_{\R^d}\Delta a(x)\squad{|x|^{-(d-\gamma_2)}*(|x|^{-\rho_{2}}| u_k|^2)}|x|^{-\rho_{2}}|u_j(x)|^2\,dx
\nonumber \\
  +\frac {2b\beta_{2}}{d-1}\sum_{j,k=1}^N\int_{\R^d}\Delta a(x)\squad{|x|^{-(d-\gamma_2)}*(|x|^{-\rho_{2}}\overline u_k u_j)  } |x|^{-\rho_{2}}u_k(x)  \overline u_j(x)\,dx,
 \nonumber
 \end{align}
with $\beta_{1}=(p-2)(d-1)+2\rho_{1}$ and $\beta_{2}=d-1+2\rho_{2}$.
\end{corollary}
\begin{proof}
Let us directly apply the Morawetz identity \eqref{eq:mor2} discarding some negative terms. We observe first that 
\begin{equation}\label{}
 \sum_{j=1}^N \int_{\R^d} (-\Delta^3a(x)+\sigma_1 \Delta^2a(x))|u_j(x)|^{2}\,dx\leq 0,
\end{equation}
from \eqref{eq.delta2} and \eqref{eq.delta3}.  Then we get
\begin{align*}
-\frac{2(p-2)}{p}\sum _{j,k=1}^Nb_{jk}
\int_{\R^d} \Delta a(x)\squad{|x|^{-\rho_{1}}\round{|x|^{-(d-\gamma_1)}*(|x|^{-\rho_{1}}|u_k|^p)}} | u_j(x)|^{p}\,dx
\\
+
\frac 4 p\sum_{j,k=1}^N b_{jk}\int_{\R^d}\nabla a(x)\cdot \nabla \squad{|x|^{-\rho_{1}}\round{|x|^{-(d-\gamma_1)}*(|x|^{-\rho_{1}}| u_k|^p)}} | u_j(x)|^{p}\,dx\\
=-\sum_{j,k=1}^N\frac{2(p-2)}{p} b_{jk}\int_{\R^d}\frac{d-1}{|x|}\ \squad{|x|^{-(d-\gamma_1)}*(|x|^{-\rho_{1}}| u_k(x)|^p)}|x|^{-\rho_{1}}|u_j(x)|^{p}\,dx&
 \nonumber\\
-\sum_{j,k=1}^N\frac {4b_{jk}}p\int_{\R^d} \frac{\rho_{1}}{|x|} |x|^{-\rho_{1}}\squad{|x|^{-(d-\gamma_1)}*(|x|^{-\rho_{1}}| u_k|^p)}|u_j(x)|^{p}\,dx&
 \nonumber\\
 +\frac 14\sum_{\substack{j,k=1}}^Nb_{jk}^*\int_{\R^{2d}}
 \frac{|x|^{-\rho_{1}}|y|^{-\rho_{1}}}{|x-z|^{d-\gamma_1+2}}|u_j(x)|^{p}| u_k(z)|^{p}(x-z) \cdot \round{\frac{x}{|x|}-\frac{z}{|z|}}\,dxdz
 \nonumber\\
 \leq 
 -\frac{2\beta_{1}}{(d-1)p}\sum _{j,k=1}^Nb_{jk}
\int_{\R^d}\Delta a(x)\squad{|x|^{-\rho_{1}}\round{|x|^{-(d-\gamma_1)}*(|x|^{-\rho_{1}}|u_k|^p)}} | u_j(x)|^{p}\,dx,
\nonumber
\end{align*}
wieh $b_{jk}^{*}$ as in \eqref{eq:nonlin2} and where the last inequality is obtained by \eqref{eq.ineq}. Then we can manage in the same way the last four terms in \eqref{eq:mor2}. In addition the first two terms are non-negative because of \eqref{eq.quad1} and \eqref{eq.quad2}, which are still fulfilled with $y=0$. This guarantees the proof of the corollary.
\end{proof}
The direct consequence of Proposition \eqref{CoLa}, Lemma \eqref{lem:tenmor} and Corollary \eqref{singmor} are the linear and nonlinear Morawetz estimates localized on the space-time slabs $\R\times B_{\tilde x}^d(r)$. We have

\begin{proposition}\label{lowcorrest}
 Assume $d\geq 3$, $p\geq 2$, $(b_{jj},\sigma_1)\neq(0,0),$ for all $j=1, \dots, N$, and $(\rho_{1},\rho_{2}, \sigma_{2})=(0,0,0)$. Let $(u_j)^N_{j=1}\in C(\R,H^2(\R^d)^N)$ be as in Proposition \ref{CoLa}. Then, selecting $a(x,y)=|x-y|$, $r>0$ and $\tilde x\in\R^d$, one retreives the following localized estimates. For $d\geq4$,
\begin{align} \label{eq:MorLoc1}
 \sum_{j,k,\ell=1}^N\widetilde b_{jk}\int_{\R}\sup_{\tilde x\in\R^d}\int_{B_{\tilde x}^d(r)^3}|u_j(t,x)|^{p}| |u_\ell(t,y)|^{2}|u_k(t,z)|^{p}
\,dxdydzdt
  \\
  +\sigma_{1}  \sum_{j,\ell=1}^N\int_{\R}\sup_{\tilde x\in\R^d}\int_{B_{\tilde x}^d(r)^2}
 |u_j(t,x)|^{2} |u_\ell(t,y)|^{2}
\,dxdydt
%  \nonumber\\
 \leq C\sum_{j=1}^N\|u_{j,0}\|^4_{H^2_x},
 \nonumber
  \end{align}
  where $ \widetilde b_{jk}=4b_{jk}(p-2)/p$. For $d=3$, 
  \begin{align} \label{eq:MorLoc2}
\sum_{j, k, \ell=1}^N\widetilde b_{jk}\int_{\R}\sup_{\tilde x\in\R}\int_{B_{\tilde x}^3(r)^2}|u_j(t,x)|^{p}|u_\ell(t,x)|^{2}|u_k(t,z)|^{p}
 \,dx dzdt 
  \\
  +\sigma_{1}  \sum_{j,\ell=1}^N\int_{\R}\sup_{\tilde x\in\R^3}\int_{B_{\tilde x}^3(r)}
 |u_j(t,x)|^{2} |u_\ell(t,x)|^{2}\, dxdt
 \leq C\sum_{j=1}^N\|u_{j,0}\|^4_{H^2_x}.
 \nonumber
  \end{align}
\end{proposition}

\begin{proof}
We treat the general frame $d\geq 3$. The interaction inequality \eqref{eq:tenmor2} and the fact that $\mathcal R(t)=0$ bear to
\begin{align}\label{eq.LocCor1}
  \sum_{j,\ell=1}^N\mathcal{ \dot{M}}_{j,\ell}(t)
  \\
  \lesssim 2\sum_{j,\ell=1}^N\int_{\R^{2d}}
  \Delta^2_xa(x,y)\round{\nabla _{x}| u_j(x)|^2\nabla_{y}|u_{\ell}(y)|^2+\sigma_{1} | u_j(x)|^2|u_{\ell}(y)|^2} \,dxdy
 \nonumber
 \\
- \sum_{j,\ell, k}^N \widetilde b_{jk}\int_{\R^{3d}}\Delta_x a(x,y) \frac{1}{|x-z|^{d-\gamma_1}}
|u_j(x)|^{p}|u_{\ell}(y)|^2| u_k(z)|^p\,dxdydz,
\nonumber
\end{align}
where in the second line we used the first in \eqref{eq.leadtherm}. This implies, after integrating in time over the interval $J=[t_{1}, t_{2}]$ with $t_{1},t_{2}\in \R$ and looking at \eqref{eq.delta1}, \eqref{eq.delta2},
\begin{align}\label{eq.LocCor2}
\sum_{j,\ell=1}^N\sup_{t\in J}|\mathcal{M}_{j,\ell}(t)|
  \gtrsim -2\sigma_{1}\sum_{j,\ell=1}^N\int_{J}\int_{\R^{2d}}
  \Delta^2_xa(x,y) | u_j(x)|^2|u_{\ell}(y)|^2 \,dxdydt
 \\
 +\sum_{j,\ell, k}^N\widetilde b_{jk}\int_{J}\int_{\R^{3d}}\Delta_x a(x,y) \frac{1}{|x-z|^{d-\gamma_1}}
|u_j(x)|^{p}|u_{\ell}(y)|^2| u_k(z)|^p\,dxdydzdt.
\nonumber
\end{align}
Both the inequalities \eqref{eq:MorLoc1} and s\eqref{eq:MorLoc2} follow finally from the bounds
\begin{align}\label{eq.lboundA}
\inf_{x,y,z\in B_{\tilde x}^d(r)}\round{\frac{1}{|x-y|}, \frac{1}{|z-y|}}=\inf_{x,y,z\in B_0^d(r)}\round{\frac{1}{|x-y|}, \frac{1}{|z-y|}}\gtrsim 1,\\
\label{eq.lboundB}
\sum_{j,\ell=1}^N\sup_{t\in J}|\mathcal{M}_{j,\ell}(t)|\lesssim \sum_{j, \ell=1}^{N}\sup_{t\in J}\|\nabla_{x}u_{j}\|_{L^{2}_{x}}^{2}\|u_{\ell}\|_{L^{2}_{x}}^{2}\lesssim \sum_{j=1}^{N}\|u_{j,0}\|_{H^{2}_{x}}^{4} ,
\end{align}
once one lets $t_{1}\rightarrow -\infty$ and $t_{2}\rightarrow \infty$.
 \end{proof}
 We have also
 
 \begin{proposition}\label{higcorrest}
 Assume $d\geq 5$, $p=2$ or $(\rho_{1},\rho_{2},\sigma_2)\neq(0,0,0)$. Let $(u_j)^N_{j=1}\in C(\R,H^2(\R^d)^N)$ be as in Proposition \ref{CoLa}. Then, selecting $a(x,y)=|x-y|$, $r>0$ and $\tilde x\in\R^d$, one gets the following localized estimates. For $d\geq6$ 
\begin{align} \label{eq:MorLoc3}
 \sum_{j,\ell=1}^N\int_{\R}\sup_{\tilde x\in\R^d}\int_{B_{\tilde x}^d(r)^2}
 |u_j(t,x)|^{2} |u_\ell(t,y)|^{2}
\,dxdydt
%  \nonumber\\
 \leq C\sum_{j=1}^N\|u_{j,0}\|^4_{H^2_x},
  \end{align}
  For $d=5$, 
  \begin{align} \label{eq:MorLoc4}
\sum_{j, \ell=1}^N\int_{\R}\sup_{\tilde x\in\R}\int_{B_{\tilde x}^5(r)}  |u_j(t,x)|^{2} |u_\ell(t,x)|^{2}
 \,dxdt \leq C\sum_{j=1}^N\|u_{j,0}\|^4_{H^2_x}.
  \end{align}
\end{proposition}
\begin{proof}
By \eqref{eq:tenmor2} and removing the non-positive terms appearing in \eqref{eq.LocCor1}  we can write
\begin{align}\label{eq.LocCor3}
  \sum_{j,\ell=1}^N\mathcal{ \dot{M}}_{j,\ell}(t)
  \lesssim
  2\sum_{j,\ell=1}^N\int_{\R^{2d}}
 - \Delta^3_xa(x,y)| u_j(x)|^2|u_{\ell}(y)|^2 \,dxdy+\mathcal R(t),
\end{align}
where we applied now the second in \eqref{eq.leadtherm}. Then
 because of \eqref{eq.delta1}, \eqref{eq.delta2}, indeed one attains
\begin{align}\label{eq.LocCor4}
\int_{J} |\mathcal R(t)| \, dt +\sum_{j,\ell=1}^N\sup_{t\in J}|\mathcal{M}_{j,\ell}(t)|
\\
  \gtrsim\sum_{j,\ell=1}^N\int_{J}\int_{\R^{2d}}
  \Delta^3_xa(x,y) | u_j(x)|^2|u_{\ell}(y)|^2 \,dxdydt,
  \nonumber
\end{align}
with $J$ as above. It turns up also, by \eqref{eq.dMor} and $\nabla_x a(x,y)\lesssim 1,$ that
\begin{align}\label{eq.LocCor5}
\int_{J} |\mathcal R(t)| \, dt 
\\
\lesssim 2\sigma_2\sum_{j,\ell=1}^N \int_{\R^{2d}}  |\nabla V(x)|| u_j(x)|^2|u_{\ell}(y)|^2\,dxdy&
 \nonumber
 \\
+ \sum_{j, k, \ell}^Nb_{jk}\int_{J}\int_{\R^{2d}} \frac 1{|x|}\squad{|x|^{-(d-\gamma_1)}* (|x|^{-\rho_{1}}| u_k(x)|^p)}|x|^{-\rho_{1}}|u_j(x)|^{p}|u_{\ell}(y)|^2\,dxdydt
\nonumber
\nonumber\\
+2b\sum_{j,k, \ell=1}^N\int_{\R^{2d}}\frac 1{|x|}\left(\squad{|x|^{-(d-\gamma_2)}*(|x|^{-\rho_{2}}| u_k|^2 )} |x|^{-\rho_{2}} |u_j(x)|^2|u_{\ell}(y)|^2\,\right.
\nonumber
 \\
 \left. - \squad{|x|^{-(d-\gamma_2)}*({|x|^{-\rho_{2}}\overline u_k u_j  )}}|x|^{-\rho_{2}} u_k(x)  \overline u_j(x) |u_{\ell}(y)|^2\right)\,dxdydt
\nonumber\\
\lesssim \sum_{j,\ell=1}^N\sup_{t\in J}|M_{j}(t)|\|u_{\ell}\|_{L^{2}_{x}}^{2}\lesssim  \sum_{j=1}^{N}\|u_{j,0}\|_{H^{2}_{x}}^{4} ,
\end{align}
where the inequalities in the last line are achieved by \eqref{eq.singmor} and \eqref{eq.singMor}. The above estimate combined with \eqref{eq.LocCor4}, \eqref{eq.delta3} and the bounds \eqref{eq.lboundA}, \eqref{eq.lboundB} guarantees \eqref{eq:MorLoc3} and \eqref{eq:MorLoc4}. 
 \end{proof}

\section{Decay in energy space: proof of Theorem \ref{thm:desl}}\label{DecSolu}
The proof of  the main theorem concerning the dacay of the solutions to HFC4 equations in the energy space is given in this section.

% \begin{remark}
%   In fact something more can be proved: the system \eqref{eq:HFC4} is \emph{well posed}
%   in the sense of Definition 3.1.5, \cite{Ca} (see also  Theorem 4.4.6). 
% \end{remark}
{\bf Proof of Theorem \ref{thm:desl}.}
We prove \eqref{eq:desol0} for a suitable $2<r<2d/(d-4)$ if $d\geq 5$ (for $2<r<+\infty$, if $d=3,4$), the result for the general case will follow by the conservation of mass \eqref{eq:mcon}, the kinetic energy 
\eqref{eq:econs} and interpolation. We want to prove then 
\begin{equation}\label{eq:potenergy2}
\lim_{t\rightarrow \pm \infty} \|u_{j}(t)\|_{L^{\frac{2d+4}{d}}_x}=0,
\end{equation}
for all $j=1,\dots, N$. Moreover we focus on forward times only, because
the backwards can be managed similarly. 
Assume that \eqref{eq:potenergy2} does not hold (see \cite{CT} and \cite{Vis}). Hence we can find a sequence
$\{t_n\}$ with $t_n \to +\infty$ and a $\delta>0$ such that
\begin{equation}\label{eq:escseq}
 \inf_n \|u_{j}(t_n, x)\|_{\Lin^{\frac{2d+4}{d}}_x}=\delta,
\end{equation}
for some $j\in\{1, \dots, N\}$. The Gagliardo-Nirenberg inequality (see \cite{TeTzVi})
\begin{equation}\label{eq:GNloc3}
\|u_{j}\|_{L^{\frac{2d+4}{d}}_x}^{\frac{2d+4}{d}}\lesssim \left(\sup_{x\in \R^d} \|u_{j}\|_{L^2(B^d_{\widetilde x}(1))}\right)^{\frac{4}{d}} 
\|u_{j}\|^2_{H^2_x},
\end{equation}
allows us to say that there exists $x_n \in \R^d$ and a $\varepsilon_0>0$ such that
\begin{equation}\label{eq:seqspac}
\inf_{n} \|u_j(t_n, x)\|_{\Lin^2(B_{x_n}^d(1))}=\varepsilon_0.
\end{equation}
Fix a cut-off function $\phi(x)\in C^\infty_0(\R^d)$, so as
$\phi(x)=1$ for $B^d_{0}(1)$ and $\phi(x)=0$ for $x\notin B^d_{0}(2)$.
Then by choosing
$a(x)=\phi (x-x_n)$ we get from the relation
\begin{align*}
\partial_{t}|u_{j}|^{2}=-2 \nabla_{x} \Im\left(\overline{u}_{j}\nabla_{x}(\Delta_{x}-\sigma_{1}) u_{j}\right)+2 \nabla_{x}\Im\left(\nabla_{x} \overline{u}_{j} D^{2}_{x} u_{j}\right),
\end{align*}
the following
\begin{equation*}
\begin{split}
\left|\frac d{dt} \int_{\R^d} \phi (x -x_n)  | u_{j}(t,x)|^2 dx \right|\lesssim\sigma_{1}\left|\int_{\R^d} \nabla_x\phi (x -x_n) \Im(\nabla_xu_{j}(t,x)\overline u_{j}(t,x))dx\right|\\
+\left|\int_{\R^d} \Delta_x\phi (x -x_n) \Im(\Delta_xu_{j}(t,x)\overline u_{j}(t,x))dx\right|+\\
+\left|\int_{\R^d} \nabla_x\phi (x -x_n) \Im(\Delta_xu_{j}(t,x)\nabla_x\overline u_{j}(t,x)) dx\right|
\lesssim \sup_t \|u_{j}(t,x)\|_{H^2_x}^2.
\end{split}
\end{equation*}
Consequently, the fundamental theorem of calculus implies 

\begin{equation}
\left|\int_{\R^d} \varphi(x -x_n) |u_j(w_{2}, x)|^2 dx - \int_{\R^d} \varphi(x -x_n) |u_j(w_{1}, x)|^2 dx\right|\leq \widetilde C |w_{1}-w_{2}|,
\end{equation}
for a $ \widetilde C>0$ independent by $n$. Choosing $w_{1}=t_n$ and $w_{2}=t_{n}+ t^*$, with $ t^*>0$,  we infer
\begin{equation}
\int_{\R^d} \varphi(x -x_n) |u_j(t_{n}+ t^*, x)|^2 dx\geq  \int_{\R^d} \varphi(x -x_n) |u_j(t_n, x)|^2 dx - \widetilde Ct^{*},
\end{equation}
which reads, since the support property of $\phi$, 

\begin{equation}
\int_{B_{x_n}^d(2)} |u_j(t_{n}+t^{*}, x)|^2 dx\geq  \int_{B_{x_n}^d(1)} |u_j(t_n, x)|^2 dx - \widetilde Ct^{*}.
\end{equation}
By the previous estimates and \eqref{eq:seqspac}, we inherit  
\begin{equation}\label{eq:infbound1}
 \|u_j(t, x)\|_{\Lin^2(B_{x_n}^d(2))}\geq \varepsilon_1,
\end{equation}
with $\varepsilon_{1}>0$, for all  $t\in (t_n, t_n+ t^*)$, provided that $ t^*>0$ is suitable small in order also to come by the time intervals disjoint.
Furthermore, by means of the lower bound \eqref{eq:infbound1} we have 
\begin{equation}\label{eq:infbound2}
\inf_{n}\round{\inf_{t\in (t_n, t_n+t^*)} \sum_{j=1}^N \|u_j(t)\|^{2}_{L^{2}_x(B_{x_n}^d(2))}} \gtrsim \varepsilon^{2}_1>0.
\end{equation}
Let us now complete the proof by distinguishing several different cases. 
\subsubsection*{Let be $d\geq 3$, $p> 2$, $(b_{jj},\sigma_1)\neq(0,0),$ for all $j=1, \dots, N$, $(\rho_{1},\rho_{2})=(0,0)$ and $\sigma_2=0.$\\
Case $d\geq 4$}  We carry out, by H\"older inequality and \eqref{eq:infbound2},
\begin{align*} %\label{eq:infbound3}
\sum_{j, k, \ell=1}^N\widetilde b_{jk}\int_{\R}\sup_{\tilde x\in\R^d}\int_{B_{\tilde x}^d(2)^3}|u_j(t,x)|^{p}|u_\ell(t,y)|^{2}|u_k(t,z)|^{p}
 \,dxdydzdt 
   \nonumber \\
    +\sigma_{1}  \sum_{j,\ell=1}^N\int_{\R}\sup_{\tilde x\in\R^d}\int_{B_{\tilde x}^d(r)^2}
 |u_j(t,x)|^{2} |u_\ell(t,y)|^{2}  \,dxdydzdt 
\nonumber \\
   \geq \inf_{j} \widetilde b_{jj} \int_{\R} \round{\int_{B_{x_n}^d(2)}|u_j(t,x)|^{2}
 \,dx}^{\frac{2p+2}2}dt +  \sigma_{1} \int_{\R} \round{\int_{B_{x_n}^d(2)}|u_j(t,x)|^{2}
 \,dx}^{2}dt 
 	\nonumber\\
 \gtrsim    \sum_{n} \int_{t_n}^{t_n+t^*} (\varepsilon^{2p+2}_1+\sigma_{1}\varepsilon^{4}_1)\,dt\gtrsim \sum_{n} t^*(\varepsilon^{2p+2}_1+\sigma_{1}\varepsilon^{4}_1)=\infty,&
\end{align*}
which is in contradiction with \eqref{eq:MorLoc1}. \\
\subsubsection*{Case $d=3$} Arguing in a similar manner as aforesaid, by an application of the H\"older inequality, one obtain the inequality
\begin{equation} 
\begin{split}
\sum_{j, k, \ell=1}^N\widetilde b_{jk}\int_{\R}\sup_{\tilde x\in\R^d}\int_{B_{\tilde x}^3(2)^2}|u_j(t,x)|^{p}|u_\ell(t,
x)|^{2}|u_k(t,z)|^{p}
 \,dxdzdt 
   \nonumber \\
    +\sigma_{1}  \sum_{j,\ell=1}^N\int_{\R}\sup_{\tilde x\in\R^d}\int_{B_{\tilde x}^3(r)}
 |u_j(t,x)|^{2} |u_\ell(t,x)|^{2}  \,dxdt 
  \nonumber\\
  \gtrsim \sum_{n} t^*(\varepsilon^{2p+2}_1+\sigma_{1}\varepsilon^{4}_1)=\infty,
\end{split}
\end{equation}
which denies \eqref{eq:MorLoc2}. 
\subsubsection*{Let be $p=2$ or  $(\rho_{1},\rho_{2},\sigma_{1})\neq(0,0,0)$.\\
Case $d\geq 5$} We will focus on $d\geq 6$ and utilize Proposition \ref{higcorrest}. Then it is possible to accomplish 
\begin{equation*} 
\begin{split}
 \sum_{j, \ell=1}^N\int_{\R}\sup_{\tilde x\in\R^d}\int_{B_{\tilde x}^d(2)^{2}}|u_j(t,x)|^{2}|u_\ell(t,y)|^{2}
 \,dxdydt   \\
  \gtrsim  \sum_{j,\ell=1}^N\sum_{n} \int_{t_n}^{t_n+t^*}\int_{B_{ x_n}^d(2)^{2}}|u_j(t,x)|^{2}|u_\ell(t,y)|^{2}
  \,dxdydt
  \nonumber\\
   \gtrsim   \sum_{n} \int_{t_n}^{t_n+t^*} \varepsilon^{4}_1\,dt= \sum_{n} t^*\varepsilon^{4}_1\,dt=\infty,
\end{split}
\end{equation*}
which is in conflict with \eqref{eq:MorLoc3}. We skip the case $d=5$ because it can be arranged alike by taking advantage of \eqref{eq:MorLoc4}. Hence the proof is now complete.\\

 \section{Scattering for the HFC4}\label{NFC4scat}
We want  to prove Theorem \ref{thm:scattHFC4}. We acquire first the essential space-time summability for the solutions to \eqref{eq:HFC4}, then we display the scattering. We have the following
 \begin{lemma}\label{StriHFC4}
 Assume $(u_j)_{j=1}^N\in  C(\R,H^2(\R^d)^N)$ as in Theorem \ref{thm:scattHFC4}. We have
 \begin{align}\label{eq:StriHFC4}
 (u_j)_{j=1}^N\in L^q(\R, W^{2,r}(\R^d)^N),
 \end{align}
for every pair $(q,r)\in \mathcal B$.
\end{lemma}
\begin{proof}
We consider  the integral operator associated to \eqref{eq:HFC4}, that is
\begin{align}\label{eq:opintHFC4}
\sum_{j=1}^{N} u_{j}(t)=e^ {it  \Delta_x}u_{j,0} +\sum_{j,k=1}^{N}  \int_{0}^{t} e^ {i(t-\tau)(\Delta^2_x-\sigma_{1}\Delta_{x}+\sigma_{2} V)}  \mathcal F(\tau, u_j , u_k) d\tau,
\end{align}
moreover, for $t_{0}>0$ we introduce the auxiliary spaces
\begin{align}\label{eq.1st2}
   \norm{u}{ X_{(t_{0}, \infty)}} =\sup_{(q,r)\in \mathcal B} \left\{\norm{u}{L^{q} ((t_{0}, \infty), L^{r}_x)}\right\}, 
   \\
   \label{eq.1st3}
     \norm{u}{\widetilde X_{(t_{0}, \infty)}} =\inf_{(\tilde q' ,\tilde r')\in \mathcal B^{'}} \left\{\norm{u}{L^{\tilde q'}((t_{0}, \infty), L^{\widetilde r'}_x)}\right\}.
\end{align}
Let us start by dealing with $p\geq 2$, $(b_{jj},\sigma_1)\neq(0,0),$ for all $j=1, \dots, N$ and $(\rho_{1},\rho_{2})=(0,0)$. We will 
restrict to $d\geq 3$ if $p>2$ and $b=0, \sigma_{2}=0$, otherwise to $d\geq 5$.
\\
\\
In order to get the property \eqref{eq:StriHFC4}, we choose now
$( q_1', r_1')$ so that 
\begin{equation}\label{eq:pairHF1}
 ( q_1,r_1):= \left(\frac{8p}{dp-d-\gamma_1},\frac{2dp}{d+\gamma_1}\right)\in \mathcal {B}.
\end{equation}
In this way, Strichartz estimates \eqref{eq:StriBiharm}, H\"older and Hardy-Littlewood-Sobolev inequalities bring to 
\begin{align}\label{eq.nnl1}
\sum_{s=0}^1\left\|\sum_{j,k=1}^N b_{jk}\Delta_{x}^{s}\round{ \squad{|x|^{-(d-\gamma_1)}*| u_k|^p} | u_{j}|^{p-2}  u_j}\right\|_{\widetilde X_{(t_{0}, \infty)}}\\
\lesssim\sum_{s=0}^1\left\|\sum_{j,k=1}^N b_{jk}\Delta_{x}^{s}\round{ \squad{|x|^{-(d-\gamma_1)}*|u_k|^p} | u_j|^{p-2}  u_j}\right\|_{L^{q_1'}((t_{0}, \infty), L^{r_1'}_x)}
\nonumber\\
\lesssim \sum_{s=0}^1\norm{ \sum^N_{j,k=1}b_{jk}\|\Delta_{x}^{s}u_j\|_{L^{r_1}_x}\|u_k\|_{L_x^{r_1}}^{p} \|u_j\|_{L_x^{r_1}}^{p-2}}{L^{q_1'}((t_{0}, \infty))}
\nonumber\\
+\sum_{s=0}^1\left\|\sum_{j,k=1}^N b_{jk}  \squad{|x|^{-(d-\gamma_1)}*\Delta_{x}^{s}| u_k|^p} | u_j|^{p-2}  u_j\right\|_{L^{q_1'}((t_{0}, \infty), L^{r_1'}_x)}
\nonumber\\
\lesssim \sum_{s=0}^1\sum_{j,k=1}^N b_{jk}\norm{\|\Delta_{x}^{s}u_j\|_{L^{r_1}_x}\|u_k\|_{L_x^{r_1}}^{p} \|u_j\|_{L_x^{r_1}}^{p-2}+\|\Delta_{x}^{s} u_k\|_{L_x^{r_1}}\|u_k\|_{L_x^{r_1}}^{p-1}\|u_j\|_{L_x^{r_1}}^{p-1}}{L^{q_1'}((t_{0}, \infty))}
\nonumber\\
\lesssim\sum_{s=0}^1 \sum_{j,k=1}^N\Big\|
    \|\Delta_{x}^{s}u_{k}\|_{L^{r_1}_x}
   \|u_{j}\|_{L_x^{r_1}}^{(2p-2)}\Big\|_{L^{q_1'}((t_{0}, \infty))}.
\nonumber
\end{align} 
Here we utlized the estimate given in \cite{Kat}, that is, if $\nu \geq 1$, then for $1<r_{1}, r_{2}<\infty$ and $1<r_{1} \leqslant \infty$ 
\begin{equation}\label{eq:HiOrDer}
\left\|\Delta_{x}\left(|u|^{\nu} u\right)\right\|_{L_{x}^{r}} \lesssim\|u\|_{L_{x}^{r_{1}}}^{\nu}\|\Delta_{x} u\|_{L_{x}^{r_{2}}}, \qquad \frac{1}{r}=\frac{\nu}{r_{1}}+\frac{1}{r_{2}}.
\end{equation}
 We take $\theta_1=(q_1-q_1')/(2pq_1'-2q'_1)$. Direct calculations unveil that 
$$
\frac 1{q_1'}=\frac{2(p-1)\theta_1+1}{q_1}, \quad 1-\theta_1=\frac {(2p-1)q'_1-q_1}{(2p-2)q_1'}
$$
and  $\theta_{1}\in(0,1)$ by virtue of \eqref{eq:bs}, \eqref{eq:bsII}. Accordingly, the last term of the inequality \eqref{eq.nnl1} is not greater than
 \begin{align}\label{eq.nnl2}
 \sum_{s=0}^1\sum_{j,k=1}^N  \Big\|  \|\Delta_{x}^{s}u_{k}\|_{L^{r_1}_x}
   \|u_{j}\|_{L_x^{r_1}}^{(2p-2)(1-\theta_1)}\|u_{j}\|_{L_x^{r_1}}^{(2p-2)\theta_1}\Big\|_{L^{q_1'}((t_{0}, \infty))}
   \\
    \lesssim 
   \sum_{j,k=1}^N  \Big\|  \|(1-\Delta_{x})u_{k}\|_{L^{r_1}_x} 
   \|u_{j}\|_{L_x^{r_1}}^{(2p-2)(1-\theta_1)}\|(1-\Delta_{x})u_{j}\|_{L_x^{r_1}}^{(2p-2)\theta_1}\Big\|_{L^{q_1'}((t_{0}, \infty))}
 \nonumber \\
   \lesssim 
    \sup_{j=1,\dots, N} \|u_{j}\|_{L^{\infty}((t_{0}, \infty), L^{r_1}_x)}^{(2p-2)(1-\theta_1)} \sum_{j,k=1}^N  \Big\|  \|(1-\Delta_{x})u_{k}\|_{L^{r_1}_x} 
   \|(1-\Delta_{x})u_{j}\|_{L_x^{r_1}}^{(2p-2)\theta_1}\Big\|_{L^{q_1'}((t_{0}, \infty))}
  \nonumber \\
    \lesssim \sup_{j=1,\dots, N} \|u_{j}\|_{L^{\infty}((t_{0}, \infty), L^{r_1}_x)}^{(2p-2)(1-\theta_1)}\round{\sum_{k=1}^N 
  \|(1-\Delta_{x})u_{k}\|_{L^{q_1}((t_{0}, \infty), L^{r_1}_x)}}^{(2p-2)\theta_1+1},
  \nonumber
\end{align}
where in the second inequality we used the Sobolev embedding and the fact that $\Delta_{x}/(1-\Delta_{x})$ is a pseudodifferential operator of order $0$, wihch is $L^{r}_{x}$-bounded, for $1<r<\infty$. We single out $( q_2', r_2')$ by taking $p=2$ in \eqref{eq:pairHF1}, that is 
\begin{equation*}
 ( q_2,r_2):= \left(\frac{8}{d-\gamma_1},\frac{4d}{d+\gamma_1}\right)\in \mathcal {B},
\end{equation*}
hence we get, analogously as above, 
\begin{align}\label{eq.nnl3}
\sum_{s=0}^1 \sum^N_{j,k=1}b\left\| \Delta_{x}^{s}\round{ \squad{|x|^{-(d-\gamma_2)}*| u_k|^2}    u_j-\squad{|x|^{-(d-\gamma_2)}*\overline u_{k} u_j  } u_k}\right\|_{\widetilde X_{(t_{0}, \infty)}}
\nonumber\\
\lesssim\sum_{s=0}^1 \sum^N_{j,k=1}b\left\| \Delta_{x}^{s}\round{ \squad{|x|^{-(d-\gamma_2)}*| u_k|^2}    u_j-\squad{|x|^{-(d-\gamma_2)}*\overline u_{k} u_j  } u_k}\right\|_{L^{q_2'}((t_{0}, \infty), L^{r_2'}_x)}
\nonumber\\
\lesssim\sum_{s=0}^1 \sum^N_{j,k=1}\Big\|
    \|\Delta_{x}^{s}u_{k}\|_{L^{r_2}_x}
   \|u_{j}\|_{L_x^{r_2}}^{2}\Big\|_{L^{q_2'}((t_{0}, \infty))}
\nonumber\\
   \lesssim 
 \sum^N_{j,k=1} \Big\|  
   \|u_{j}\|_{L_x^{r_2}}^{2(1-\theta_2)}\|(1-\Delta_{x})u_{k}\|_{L^{r_2}_x}^{2\theta_2+1}\Big\|_{L^{q_2'}((t_{0}, \infty))}
  \nonumber \\
     \lesssim \sup_{j=1,\dots, N} \|u_{j}\|_{L^{\infty}((t_{0}, \infty), L^{r_2}_x)}^{2(1-\theta_2)}\round{\sum_{k=1}^N 
  \|(1-\Delta_{x})u_{k}\|_{L^{q_2}((t_{0}, \infty), L^{r_2}_x)}}^{2\theta_2+1},
\end{align} 
where $\theta_2=(q_2-q_2')/2q_2'\in (0,1)$ and such that
$$
\frac 1{q_2'}=\frac{2\theta_1+1}{q_2}.
$$
An use of  \eqref{eq.nnl1}, \eqref{eq.nnl2} and \eqref{eq.nnl3}
leads to
\begin{align}\label{eq.finest1}
\sum_{j=1}^N\|(1-\Delta_{x})u_{j}\|_{X_{(t_{0}, \infty)}}
\lesssim \sum_{j=1}^N\|u_{j,0}\|_{H^2_x}
\\
+ \sup_{j=1,\dots, N}\|u_{j}\|_{L^{\infty}((t_{0}, \infty), L^{r_1}_x)}^{(2p-2)(1-\theta_1)}\round{\sum_{k=1}^N 
  \|(1-\Delta_{x})u_{k}\|_{ X_{(t_{0}, \infty)}}}^{(2p-2)\theta_1+1}
 \nonumber \\
  +\sup_{j=1,\dots, N} \|u_{j}\|_{L^{\infty}((t_{0}, \infty), L^{r_2}_x)}^{2(1-\theta_2)}\round{\sum_{k=1}^N 
  \|(1-\Delta_{x})u_{k}\|_{ X_{(t_{0}, \infty)}}}^{2\theta_2+1},
  \nonumber
\end{align}
with 
$$
\lim _{t_{0}\rightarrow +\infty}\round{\|u_{j}\|_{L^{\infty}((t_{0}, \infty), L^{r_1}_x)}+ \|u_{j}\|_{L^{\infty}((t_{0}, \infty), L^{r_2}_x)}}=0,
$$
for any $j=1,\dots,N$, by Theorem \ref{thm:desl}. Then, picking up $t_{0}$ sufficiently large we earn
$(u_j)_{j=1}^N\in L^{q}((t_{0},+\infty), W^{2,r}(\R^d)^N).$ As well, one can  infer $(u_j)_{j=1}^N\in L^{q_1}((-\infty, -t_{0}), W^{2,r_1}(\R^d)^N)$ and in conclusion, by continuity we have  $(u_j)_{j=1}^N\in L^q(\R, W^{2,r}(\R^d)^N)$.\\
\\
Let us manage $d\geq 5$, for $p=2$ or $(\rho_{1},\rho_{2}, \sigma_2) \neq (0,0,0)$.
\\
\\
We choose now the biharmonic-admissible pairs
\begin{equation}\label{eq:pairHF3}
 ( q_3,r_3):= \left(\frac{8p}{dp-d-\gamma_1+2\rho_{1}},\frac{2dp}{d+\gamma_1-2\rho_{1}}\right),
\end{equation}
\begin{equation}\label{eq:pairHF6}
 ( q_4,r_4):= \left(\frac{8(2p-1)}{d(2p-1)-(d+4+2 \gamma_{1}-4 \rho_{1})},\frac{2d(2p-1)}{d+4+2\gamma_{1}-4\rho_{1}}\right)
\end{equation}
and
\begin{equation}\label{eq:pairHF3a}
 ( q_5,r_5):= \left(\frac{8(2p-1)}{d(2p-1)-(d+6+2 \gamma_{1}-4 \rho_{1})},\frac{2d(2p-1)}{d+6+2\gamma_{1}-4\rho_{1}}\right).
\end{equation}
Again we have similarly as above, for $i=3,4, 5$,
\begin{equation}\label{eq:pairHF4}
\frac{d-4}{2d}<\frac 1{r_{i}}<\frac 12, \quad \frac 1{r_{4}}, \frac 1{r_{5}}>\frac 1d,
\end{equation}
with the last inequalities valid for $d\geq 5$ and $(2p-1)q_{i}'>q_{i}$ which is equivalent  to
\begin{equation}\label{eq:pairHF5}
\frac 1{q_{i}}>\frac 1{2p},
\end{equation}
always satisfied by means of \eqref{eq:bs}, \eqref{eq:bsII} and \eqref{eq:bsIII}. Additionally, if we pick up suitable $\epsilon_{i}>0$ such that $\epsilon_{i}\rightarrow 0$, we get by a continuity argument that
\begin{align}\label{eq.index}
\frac 1{r^{\pm}_{i}(\epsilon_{i})}=\frac1{r_{i}\pm \epsilon_{i}}=\frac1{r_{i}}\mp \frac {\epsilon_{1}}{r_{i}(r_{i}\pm \epsilon_{i}) },\quad \frac 1{q^{\pm}_{i}(\epsilon_{i})}
=\frac 1{q_{i}}\pm  \frac {d\epsilon_{i}}{4r_{i}(r_{i}\pm\epsilon_{i}) },
\end{align}
fulfill the same bounds as in \eqref{eq:pairHF4} and \eqref{eq:pairHF5}. We pursue by taking $b=0$ and concentrating on the inhomogeneous Choquard-term in the nonlinearity \eqref{eq.nonlinHF}. Namely, due  to the Strichartz estimates \eqref{eq:StriBiharm}, \eqref{eq.casStrex} and structure the space $\widetilde X_{(t_{0}, \infty)}$ we can write
\begin{align}\label{eq.nnl1inh}
\sum_{j=1}^N\|(1-\Delta_{x})(u_{j}-e^{it(\Delta^{2}_{x}-\sigma_{1}\Delta_{x}+\sigma_{2}V)}u_{j,0})\|_{X_{(t_{0}, \infty)}}\\
\lesssim
\sum_{s=0}^1\left\|\sum_{j,k=1}^N b_{jk}\round{ \squad{|x|^{-(d-\gamma_1)}*(|x|^{-\rho_{1}}\Delta_{x}^{s}| u_k|^p)} |x|^{-\rho_{1}} | u_{j}|^{p-2}  u_j}\right\|_{\widetilde X_{(t_{0}, \infty)}}
\nonumber\\
+
\sum_{s=0}^1\left\|\sum_{j,k=1}^N b_{jk}\round{ \squad{|x|^{-(d-\gamma_1)}*(|x|^{-\rho_{1}}| u_k|^p)} |x|^{-\rho_{1}} \Delta_{x}^{s}(| u_{j}|^{p-2}  u_j)}\right\|_{\widetilde X_{(t_{0}, \infty)}}
\nonumber\\
+\left\|\sum_{j,k=1}^N b_{jk}\round{ \squad{|x|^{-(d-\gamma_1)}*(\nabla_{x}|x|^{-\rho_{1}}| u_k|^p)} |x|^{-\rho_{1}} | u_{j}|^{p-2}  u_j}\right\|_{L^{2}((t_{0}, \infty), L_{x}^{\frac{2d}{d+2}})}
\nonumber\\
+\left\|\sum_{j,k=1}^N b_{jk}\round{ \squad{|x|^{-(d-\gamma_1)}*(|x|^{-\rho_{1}}| u_k|^p)} \nabla_{x}|x|^{-\rho_{1}} | u_{j}|^{p-2}  u_j}\right\|_{L^{2}((t_{0}, \infty), L_{x}^{\frac{2d}{d+2}})}
\nonumber\\
+\left\|\sum_{j,k=1}^N b_{jk}\round{ \squad{|x|^{-(d-\gamma_1)}*(\nabla_{x}|x|^{-\rho_{1}}| u_k|^p)} |x|^{-\rho_{1}} | u_{j}|^{p-2}  u_j}\right\|_{L^{2}((t_{0}, \infty), L_{x}^{\frac{2d}{d+4}})}
\nonumber\\
+\left\|\sum_{j,k=1}^N b_{jk}\round{ \squad{|x|^{-(d-\gamma_1)}*(|x|^{-\rho_{1}}| u_k|^p)} \nabla_{x}|x|^{-\rho_{1}} | u_{j}|^{p-2}  u_j}\right\|_{L^{2}((t_{0}, \infty), L_{x}^{\frac{2d}{d+4}})}.
\nonumber
\end{align} 
At this point we note that 
\begin{align}\label{eq.1st2a}
  \norm{u}{\widetilde X_{(t_{0}, \infty)}}  \leq \norm{u}{ \widetilde X_{(t_{0}, \infty)}(\{|x|\leq1\})}+ \norm{u}{ \widetilde X_{(t_{0}, \infty)}\{(|x|>1\})},
\\
  \norm{u}{L^{2}((t_{0}, \infty), L_{x}^{\frac{2d}{d+2}})} 
   \leq \norm{u}{ L^{2}((t_{0}, \infty), L_{x}^{\frac{2d}{d+2}}(\{|x|\leq1\}))}+ \norm{u}{ L^{2}((t_{0}, \infty), L_{x}^{\frac{2d}{d+2}}(\{|x|>1\}))},
   \nonumber\\
  \norm{u}{L^{2}((t_{0}, \infty), L_{x}^{\frac{2d}{d+4}})} 
   \leq \norm{u}{ L^{2}((t_{0}, \infty), L_{x}^{\frac{2d}{d+4}}(\{|x|\leq1\}))}+ \norm{u}{ L^{2}((t_{0}, \infty), L_{x}^{\frac{2d}{d+4}}(\{|x|>1\}))}
   \nonumber
\end{align}
and\footnote{The embedding $$\norm{u}{ L_{x}^{\frac{2d}{d+4}}(\{|x|\leq1\})}\lesssim \norm{u}{ L_{x}^{\frac{2d}{d+2}}(\{|x|\leq1\})},$$ shows that we need the pair $( q_5,r_5)$ only in the region $|x|>1$. Anyway we  do not pursue this approach for the aim of simplicity.}
\begin{equation}\label{eq.weight}
\begin{split}
\norm{|x|^{-\rho_{1}}}{L^{a}_{x}(|x|\leq 1)}<+\infty\quad\text{if}\quad a\rho_{1}<d,
\\
\norm{|x|^{-\rho_{1}}}{L^{a}_{x}(|x|> 1)}<+\infty\quad\text{if}\quad a\rho_{1}>d.
\end{split}
\end{equation}
\end{proof}
Thus we get, by proceeding as in the proof of \eqref{eq.finest1} and applying \eqref{eq.index}, \eqref{eq.1st2a}, \eqref{eq.weight}, that 
the second term of \eqref{eq.nnl1inh} is not greater than
\begin{align}\label{eq.finest2}
 \sup_{j=1,\dots, N}\|u_{j}\|_{L^{\infty}((t_{0}, \infty), L^{r^{+}_3}_x)}^{(2p-2)(1-\theta^{+}_3)}\round{\sum_{k=1}^N 
  \|(1-\Delta_{x})u_{k}\|_{ X_{(t_{0}, \infty)}}}^{(2p-2)\theta^{+}_3+1}
  \\
  +\sup_{j=1,\dots, N} \|u_{j}\|_{L^{\infty}((t_{0}, \infty), L^{r^{-}_3}_x)}^{(2p-2)(1-\theta^{-}_3)}\round{\sum_{k=1}^N 
  \|(1-\Delta_{x})u_{k}\|_{ X_{(t_{0}, \infty)}}}^{(2p-2)\theta^{-}_3+1}
  \nonumber\\
  + \sup_{j=1,\dots, N}\|u_{j}\|_{L^{\infty}((t_{0}, \infty), L^{r^{+}_4}_x)}^{(2p-2)(1-\theta^{+}_4)}\round{\sum_{k=1}^N 
 \norm{\frac{\nabla_{x}}{(1-\Delta_{x})^{\frac 12}}(1-\Delta_{x})^{\frac 12}u_{k}}{ X_{(t_{0}, \infty)}}}^{(2p-2)\theta^{+}_4+1}
 \nonumber \\
  +\sup_{j=1,\dots, N} \|u_{j}\|_{L^{\infty}((t_{0}, \infty), L^{r^{-}_4}_x)}^{(2p-2)(1-\theta^{-}_4)}\round{\sum_{k=1}^N 
  \norm{\frac{\nabla_{x}}{(1-\Delta_{x})^{\frac 12}}(1-\Delta_{x})^{\frac 12}u_{k}}{ X_{(t_{0}, \infty)}}}^{(2p-2)\theta^{-}_4+1}
  \nonumber\\
    \nonumber\\
  + \sup_{j=1,\dots, N}\|u_{j}\|_{L^{\infty}((t_{0}, \infty), L^{r^{+}_5}_x)}^{(2p-2)(1-\theta^{+}_4)}\round{\sum_{k=1}^N 
 \norm{\frac{\nabla_{x}}{(1-\Delta_{x})^{\frac 12}}(1-\Delta_{x})^{\frac 12}u_{k}}{ X_{(t_{0}, \infty)}}}^{(2p-2)\theta^{+}_5+1}
 \nonumber \\
  +\sup_{j=1,\dots, N} \|u_{j}\|_{L^{\infty}((t_{0}, \infty), L^{r^{-}_5}_x)}^{(2p-2)(1-\theta^{-}_4)}\round{\sum_{k=1}^N 
  \norm{\frac{\nabla_{x}}{(1-\Delta_{x})^{\frac 12}}(1-\Delta_{x})^{\frac 12}u_{k}}{ X_{(t_{0}, \infty)}}}^{(2p-2)\theta^{-}_5+1}
  \nonumber\\
 \lesssim
 \sum_{i=3}^5  \sup_{j=1,\dots, N}\|u_{j}\|_{L^{\infty}((t_{0}, \infty), L^{r^{+}_i}_x)}^{(2p-2)(1-\theta^{+}_i)}\round{\sum_{k=1}^N 
  \|(1-\Delta_{x})u_{k}\|_{ X_{(t_{0}, \infty)}}}^{(2p-2)\theta^{+}_i+1}
 \nonumber \\
  + \sum_{i=3}^5 \sup_{j=1,\dots, N} \|u_{j}\|_{L^{\infty}((t_{0}, \infty), L^{r^{-}_i}_x)}^{(2p-2)(1-\theta^{-}_i)}\round{\sum_{k=1}^N 
  \|(1-\Delta_{x})u_{k}\|_{ X_{(t_{0}, \infty)}}}^{(2p-2)\theta^{-}_i+1},
  \nonumber
\end{align}
with $\theta^{\pm}_i=(q^{\pm \prime}_i-q^{\pm }_i)/(2pq^{\prime\pm}_i-2q^{\prime\pm}_i)\in (0,1)$, where we made an use of the fact that $\nabla_{x}/(1-\Delta_{x})^{\frac12}$ is a pseudo-differential operator of order $0$\footnote{Notice that one has the operators identity  $$ \frac {\nabla_{x}}{(1-\Delta_{x})^{\frac12}}=\frac{\nabla_{x}}{|\nabla_{x}|}\frac {|\nabla_{x}|}{(1-\Delta_{x})^{\frac12}}.$$} together with the natural embedding $W_{x}^{2, r}\subset W_{x}^{1,r}$. Let us take into account  the Hartree-Fock nonlinear term, that is when $b>0.$ Let us pick up for the purpose the biharmonic-admissible pairs
\begin{equation}\label{eq:pairHF7}
 ( q_6,r_6):= \left(\frac{16}{d-\gamma_2+2\rho_{2}},\frac{4d}{d+\gamma_2-2\rho_{2}}\right),
\end{equation}
\begin{equation}\label{eq:pairHF8}
 ( q_{7},r_{7}):= \left(\frac{24}{2d-(4+2 \gamma_{2}-4 \rho_{2})},\frac{6d}{d+4+2\gamma_{2}-4\rho_{2}}\right)
\end{equation}
and
\begin{equation}\label{eq:pairHF8a}
 ( q_{8},r_{8}):= \left(\frac{24}{2d-(6+2 \gamma_{2}-4 \rho_{2})},\frac{6d}{d+6+2\gamma_{2}-4\rho_{2}}\right).
\end{equation}
By arguing in like manner as above we get the extra terms 
\begin{align}\label{eq.finest3}
 \sum_{i=6}^8  \sup_{j=1,\dots, N}\|u_{j}\|_{L^{\infty}((t_{0}, \infty), L^{r^{+}_i}_x)}^{2(1-\theta^{+}_i)}\round{\sum_{k=1}^N 
  \|(1-\Delta_{x})u_{k}\|_{ X_{(t_{0}, \infty)}}}^{2\theta^{+}_i+1}
  \\
  + \sum_{i=6}^8 \sup_{j=1,\dots, N} \|u_{j}\|_{L^{\infty}((t_{0}, \infty), L^{r^{-}_i}_x)}^{2(1-\theta^{-}_i)}\round{\sum_{k=1}^N 
  \|(1-\Delta_{x})u_{k}\|_{ X_{(t_{0}, \infty)}}}^{2\theta^{-}_i+1},
  \nonumber
\end{align}
with $\theta^{\pm}_i\in (0,1)$ defined as above. Coupling \eqref{eq.nnl1inh}, \eqref{eq.finest2} and \eqref{eq.finest3}
we arrive at
\begin{align}\label{eq.finest4}
\sum_{j=1}^N\|(1-\Delta_{x})(u_{j}-e^{it(\Delta^{2}_{x}-\sigma_{1}\Delta_{x}+\sigma_{2}V)}u_{j,0})\|_{X_{(t_{0}, \infty)}}\\
\lesssim
\sum_{i=3}^5  \sup_{j=1,\dots, N}\|u_{j}\|_{L^{\infty}((t_{0}, \infty), L^{r^{+}_i}_x)}^{(2p-2)(1-\theta^{+}_i)}\round{\sum_{k=1}^N 
  \|(1-\Delta_{x})u_{k}\|_{ X_{(t_{0}, \infty)}}}^{(2p-2)\theta^{+}_i+1}
 \nonumber \\
  + \sum_{i=3}^5 \sup_{j=1,\dots, N} \|u_{j}\|_{L^{\infty}((t_{0}, \infty), L^{r^{-}_i}_x)}^{(2p-2)(1-\theta^{-}_i)}\round{\sum_{k=1}^N 
  \|(1-\Delta_{x})u_{k}\|_{ X_{(t_{0}, \infty)}}}^{(2p-2)\theta^{-}_i+1}
  \nonumber
  \\
  +\sum_{i=6}^8  \sup_{j=1,\dots, N}\|u_{j}\|_{L^{\infty}((t_{0}, \infty), L^{r^{+}_i}_x)}^{2(1-\theta^{+}_i)}\round{\sum_{k=1}^N 
  \|(1-\Delta_{x})u_{k}\|_{ X_{(t_{0}, \infty)}}}^{2\theta^{+}_i+1}
\nonumber  \\
  + \sum_{i=6}^8 \sup_{j=1,\dots, N} \|u_{j}\|_{L^{\infty}((t_{0}, \infty), L^{r^{-}_i}_x)}^{2(1-\theta^{-}_i)}\round{\sum_{k=1}^N 
  \|(1-\Delta_{x})u_{k}\|_{ X_{(t_{0}, \infty)}}}^{2\theta^{-}_i+1},
  \nonumber
\end{align} 
where 
$$
\lim _{t_{0}\rightarrow +\infty}\sum_{i=3}^8\round{\|u_{j}\|_{L^{\infty}((t_{0}, \infty), L^{r^{+}_i}_x)}+ \|u_{j}\|_{L^{\infty}((t_{0}, \infty), L^{r^{-}_i}_x)}}=0,
$$
for any $j=1,\dots,N$, by Theorem \ref{thm:desl}. This leads again  to $(u_j)_{j=1}^N\in L^q(\R, W^{2,r}(\R^d)^N)$. This ends the proof of the Lemma.

\begin{proof}[Proof of Theorem \ref{thm:scattHFC4}]
 We exploit the proof of Theorem \ref{thm:scattHFC4} covering all the different cases in a unified fashion. We start 
 by writing $\widetilde  u(t)=e^ {-it \Delta_{x}}u(t)$ and getting then from \eqref{eq:opintHFC4}

\begin{align*}
\sum_{j=1}^{N}  \round{\widetilde u_{j}(t_2)-\widetilde  u_{j}(t_1)}=  i\sum_{j,k=1}^{N} \int_{t_1}^{t_2} e^ {-i\tau \Delta_{x}}  (\Delta^2_x-\sigma_{1}\Delta_{x}+\sigma_{2} V) \mathcal F(\tau, u_j , u_k) d\tau.
\end{align*}
An use of the Strichartz estimates \eqref{eq:StriBiharm}, \eqref{eq.casStrex}, \eqref{eq.casStrexDu} along with with \eqref{eq.SobEquiv} infer
 \begin{align}\label{eq:stric1}
\sum_{j,k=1}^{N}\norm{\int_{t_1}^{t_2} e^ {-i\tau (\Delta^2_x-\sigma_{1}\Delta_{x}+\sigma_{2} V)}  \mathcal F(\tau, u_j , u_k) d\tau}{H^{2}_x}
\\
\lesssim\sum_{j,k=1}^{N}\norm{(\Delta^2_x+\sigma_{2} V)^{\frac12}\int_{t_1}^{t_2} e^ {-i\tau (\Delta^2_x-\sigma_{1}\Delta_{x}+\sigma_{2} V)}  \mathcal F(\tau, u_j , u_k) d\tau}{L^{2}_x}
\nonumber\\
+\sum_{j,k=1}^{N}\norm{\int_{t_1}^{t_2} e^ {-i\tau (\Delta^2_x-\sigma_{1}\Delta_{x}+\sigma_{2} V)}  \mathcal F(\tau, u_j , u_k) d\tau}{L^{2}_x}
\nonumber\\
\lesssim\sum_{s=0}^1\left\|\sum_{j,k=1}^N b_{jk}\Delta_{x}^{s}\round{ \squad{|x|^{-(d-\gamma_1)}*(|x|^{-\rho_{1}}| u_k|^p)} |x|^{-\rho_{1}} | u_{j}|^{p-2}  u_j}\right\|_{\widetilde X_{(t_{1}, t_{2})}}
\nonumber\\
+\rho_{1}\sum_{s=0}^1\left\|\sum_{j,k=1}^N b_{jk}\round{ \squad{|x|^{-(d-\gamma_1)}*(\nabla^{1-s}_{x}|x|^{-\rho_{1}}| u_k|^p)}\nabla^{s} |x|^{-\rho_{1}} | u_{j}|^{p-2}  u_j}\right\|_{L^{2}((t_{1}, t_{2}), L_{x}^{\frac{2d}{d+2}})}
\nonumber\\
+b\sum_{s=0}^1 \left\| \sum^N_{j,k=1}\Delta_{x}^{s}\round {\squad{|x|^{-(d-\gamma_2)}*|x|^{-\rho_{2}}| u_k|^2  } |x|^{-\rho_{2}}u_j}\right\|_{\widetilde X_{(t_{1}, t_{2})}}
\nonumber\\
+b\sum_{s=0}^1 \left\|\sum^N_{j,k=1} \Delta_{x}^{s}\round {\squad{|x|^{-(d-\gamma_2)}* |x|^{-\rho_{2}}\overline u_k u_j  } |x|^{-\rho_{2}} u_k}\right\|_{\widetilde X_{(t_{1}, t_{2})}}
\nonumber\\
+\rho_{2}b\sum_{s=0}^1 \left\| \sum^N_{j,k=1}\round {\squad{|x|^{-(d-\gamma_2)}*(\nabla_{x}^{1-s}|x|^{-\rho_{2}}| u_k|^2)  } \nabla_{x}^{s}|x|^{-\rho_{2}}u_j}\right\|_{L^{2}((t_{1}, t_{2}), L_{x}^{\frac{2d}{d+2}})}
\nonumber\\
+\rho_{2}b\sum_{s=0}^1 \left\|\sum^N_{j,k=1}\round {\squad{|x|^{-(d-\gamma_2)}* (\nabla_{x}^{1-s}|x|^{-\rho_{2}}\overline u_k u_j  }\nabla_{x}^{s} |x|^{-\rho_{2}} u_k}\right\|_{L^{2}((t_{1}, t_{2}), L_{x}^{\frac{2d}{d+2}})}
\nonumber\\
+\rho_{2}b\sum_{s=0}^1 \left\| \sum^N_{j,k=1}\round {\squad{|x|^{-(d-\gamma_2)}*(\nabla_{x}^{1-s}|x|^{-\rho_{2}}| u_k|^2)  } \nabla_{x}^{s}|x|^{-\rho_{2}}u_j}\right\|_{L^{2}((t_{1}, t_{2}), L_{x}^{\frac{2d}{d+4}})}
\nonumber\\
+\rho_{2}b\sum_{s=0}^1 \left\|\sum^N_{j,k=1}\round {\squad{|x|^{-(d-\gamma_2)}* (\nabla_{x}^{1-s}|x|^{-\rho_{2}}\overline u_k u_j  }\nabla_{x}^{s} |x|^{-\rho_{2}} u_k}\right\|_{L^{2}((t_{1}, t_{2}), L_{x}^{\frac{2d}{d+4}})}.
\nonumber
\end{align}
Then it is sufficient to show that
 \begin{align*}
\lim_{t_1, t_2\rightarrow \infty}\sum_{j=1}^{N}\|\widetilde  u_{j}(t_2)-\widetilde  u_{j}(t_1)\|_{H^{2}_x}=0,
\end{align*}
which is guaranteeed by \eqref{eq:stric1} once
 \begin{align*}
\lim_{t_1, t_2\rightarrow \infty}\sum_{j,k=1}^N b_{jk}\left\|\Delta_{x}^{s}\round{ \squad{|x|^{-(d-\gamma_1)}*(|x|^{-\rho_{1}}| u_k|^p)} |x|^{-\rho_{1}} | u_{j}|^{p-2}  u_j}\right\|_{\widetilde X_{(t_{1}, t_{2})}}
\\
=\lim_{t_1, t_2\rightarrow \infty}\sum_{j,k=1}^N b_{jk}\left\|\round{ \squad{|x|^{-(d-\gamma_1)}*(\nabla^{1-s}_{x}|x|^{-\rho_{1}}| u_k|^p)}\nabla^{s} |x|^{-\rho_{1}} | u_{j}|^{p-2}  u_j}\right\|_{L^{2}((t_{1}, t_{2}), L_{x}^{\frac{2d}{d+2}})}
\nonumber\\
=\lim_{t_1, t_2\rightarrow \infty}\sum^N_{j,k=1} \left\| \Delta_{x}^{s}\round {\squad{|x|^{-(d-\gamma_2)}*|x|^{-\rho_{2}}| u_k|^2  } |x|^{-\rho_{2}}u_j}\right\|_{\widetilde X_{(t_{1}, t_{2})}}
\nonumber\\
=\lim_{t_1, t_2\rightarrow \infty}\sum^N_{j,k=1} \left\| \Delta_{x}^{s}\round {\squad{|x|^{-(d-\gamma_2)}* |x|^{-\rho_{2}}\overline u_k u_j  } |x|^{-\rho_{2}} u_k}\right\|_{\widetilde X_{(t_{1}, t_{2})}}
\nonumber\\
=\lim_{t_1, t_2\rightarrow \infty}\sum^N_{j,k=1} \left\| \round {\squad{|x|^{-(d-\gamma_2)}*(\nabla_{x}^{1-s}|x|^{-\rho_{2}}| u_k|^2)  } \nabla_{x}^{s}|x|^{-\rho_{2}}u_j}\right\|_{L^{2}((t_{1}, t_{2}), L_{x}^{\frac{2d}{d+2}})}
\nonumber\\
=\lim_{t_1, t_2\rightarrow \infty}\sum^N_{j,k=1}\left\|\round {\squad{|x|^{-(d-\gamma_2)}* (\nabla_{x}^{1-s}|x|^{-\rho_{2}}\overline u_k u_j  }\nabla_{x}^{s} |x|^{-\rho_{2}} u_k}\right\|_{L^{2}((t_{1}, t_{2}), L_{x}^{\frac{2d}{d+2}})}
\nonumber\\
=\lim_{t_1, t_2\rightarrow \infty}\sum^N_{j,k=1} \left\| \round {\squad{|x|^{-(d-\gamma_2)}*(\nabla_{x}^{1-s}|x|^{-\rho_{2}}| u_k|^2)  } \nabla_{x}^{s}|x|^{-\rho_{2}}u_j}\right\|_{L^{2}((t_{1}, t_{2}), L_{x}^{\frac{2d}{d+4}})}
\nonumber\\
=\lim_{t_1, t_2\rightarrow \infty}\sum^N_{j,k=1}\left\|\round {\squad{|x|^{-(d-\gamma_2)}* (\nabla_{x}^{1-s}|x|^{-\rho_{2}}\overline u_k u_j  }\nabla_{x}^{s} |x|^{-\rho_{2}} u_k}\right\|_{L^{2}((t_{1}, t_{2}), L_{x}^{\frac{2d}{d+4}})}=0,
\end{align*}
for $s=0,1$ and that can be readily performed following the same lines of the proof of Lemma \ref{StriHFC4}. 
Then the proof of the theorem is completed.
\end{proof}
\section{Appendix}\label{appendix}
In this section we shall present an abstract result of independent interest. We will prove the equivalence between the classical 
interaction Morawetz estimates and the tensor Morawetz estimates, appeared systematically in the papers \cite{CGT07} and \cite{CGT}. This result is in fact employed in Section \ref{TMoraw} (identity \eqref{eq.laplt}) and could be useful in order to switch from one to another setting making the bilinear Morawetz inequalities a flexible tool.  
\begin{lemma}\label{eq.MorEqiv}
Let be $z_{j,\ell}(x,y)$ as in Lemma \ref{lem:tenmor}. Then one has the following identity
   \begin{align}\label{eq.EquivMorA}
2\Re \int_{\R^{2d}} \Delta_{x,y} z_{j,\ell}(x,y) 
    [(\Delta_{x,y} a(x,y)\overline z_{j,\ell}(x,y)]\,dxdy& \\
    \label{eq.EquivMorB}
    +4\Re \int_{\R^{2d}} \Delta_{x,y} z_{j,\ell}(x,y) 
    [(\nabla_{x},\nabla_{y}) a(x,y)\cdot (\nabla_{x},\nabla_{y}) \overline z_{j,\ell}(x,y)]\,dxdy& 
    \\
    =-2\int_{\R^{2d}}
 \Delta^2_xa(x,y) | u_j(x)|^2|u_{\ell}(y)|^2\,dxdy&
\nonumber\\
    -4 \int_{\R^{2d}}  \nabla_{x}u_{j}(x)D^2_{xy}a(x,y)\nabla_{x}\overline u_{j}(x)|u_{\ell}(y)|^2\,dxdy&
     \nonumber\\
 -4\int_{\R^{2d}} \nabla_{y}u_{\ell}(y)D^2_{xy}a(x,y)\nabla_{y}\overline u_{\ell}(y)|u_{\ell}(x)|^2\,dxdy&
  \nonumber\\
=2\int_{\R^{2d}}
 \Delta^2_xa(x,y) | u_j(x)|^2|u_{\ell}(y)|^2\,dxdy&
\nonumber\\
    -4 \int_{\R^{2d}}  \nabla_{x}u_{j}(x)D^2_{xy}a(x,y)\nabla_{x}\overline u_{j}(x)|u_{\ell}(y)|^2\,dxdy&
    \nonumber\\
 -4\int_{\R^{2d}} \nabla_{y}u_{\ell}(y)D^2_{xy}a(x,y)\nabla_{y}\overline u_{\ell}(y)|u_{\ell}(x)|^2\,dxdy&
  \nonumber\\
=2\int_{\R^{2d}}
 \Delta^2_xa(x,y) | u_j(x)|^2|u_{\ell}(y)|^2\,dxdy&
\nonumber\\
    -4 \int_{\R^{2d}}  \nabla_{x}u_{j}(x)D^2_{xy}a(x,y)\nabla_{x}\overline u_{j}(x)|u_{\ell}(y)|^2\,dxdy&
    \nonumber\\
 -4\int_{\R^{2d}} \nabla_{y}u_{\ell}(y)D^2_{xy}a(x,y)\nabla_{y}\overline u_{\ell}(y)|u_{\ell}(x)|^2\,dxdy&
 \nonumber\\
  -8 \int_{\R^{2d}}
  \Im(\overline u_{j}(x)\nabla_{x}u_{j}(x))D^2_{xy}a(x,y)
  \Im(\overline u_{\ell}(y)\nabla_{y}u_{\ell}(y))\,dxdy.&
  \nonumber
   \end{align}

\end{lemma}
\begin{proof}
A computation of $ \Delta_{x,y} z_{j,\ell}(x,y)$ yields

\begin{align}\label{eq.MorEquiv4}
 \eqref{eq.EquivMorA}=2\Re \int_{\mathbb{R}^{2d}} \Delta_{x,y} a(x,y) \overline{u}_{j}(x) \Delta_{x} u_{j}(x)|u_{\ell}(y)|^{2} \,dxdy 
\\
+2\Re \int_{\mathbb{R}^{2d}} \Delta_{x,y} a(x,y) \overline{u}_{\ell}(y) \Delta_{y}u_{\ell}(y)|u_{j}(x)|^{2} \,dxdy
\nonumber\\
=2\Re \int_{\mathbb{R}^{2d}} \Delta_{x} a(x,y) \overline{u}_{\ell}(y) \Delta_{y}u_{\ell}(y)|u_{j}(x)|^{2} \,dxdy
\nonumber\\
+2\Re \int_{\mathbb{R}^{2d}} \Delta_{y} a(x,y) \overline{u}_{\ell}(y) \Delta_{y}u_{\ell}(y)|u_{j}(x)|^{2} \,dxdy
\nonumber\\
+2\Re \int_{\mathbb{R}^{2d}} \Delta_{x} a(x,y) \overline{u}_{j}(x) \Delta_{x} u_{j}(x)|u_{\ell}(y)|^{2} \,dxdy
\nonumber\\
+2\Re \int_{\mathbb{R}^{2d}} \Delta_{y} a(x,y) \overline{u}_{j}(x) \Delta_{x} u_{j}(x)|u_{\ell}(y)|^{2} \,dxdy.
\nonumber\\
\end{align}
In addition we get
\begin{align}
\eqref{eq.EquivMorB}
\\
=4 \Re \int_{\mathbb{R}^{2d}}\nabla_{x} a(x,y)\nabla_{x} \overline{u}_{j}(x) \overline{u}_{\ell}(y)\left(\Delta_{x} u_{j}(x) u_{\ell}(y)+u_{j}(x) \Delta_{y} u_{\ell}(y)\right) \,dxdy 
\nonumber\\
+4 \Re \int_{\mathbb{R}^{2d}}\overline{u}_{j}(x) \nabla_{y} a(x,y)\nabla_{y}  \overline{u}_{\ell}(y)\left(\Delta_{x} u_{j}(x) u_{\ell}(y)+u_{j}(x) \Delta_{y} u_{\ell}(y)\right) \,dxdy 
\nonumber\\
=4 \Re \int_{\mathbb{R}^{2d}}\nabla_{x} a(x,y)\nabla_{x} \overline{u}_{j}(x) \Delta_{x} u_{j}(x)|u_{\ell}(y)|^{2}  \,dxdy 
\nonumber\\
4 \Re \int_{\mathbb{R}^{2d}} \nabla_{x} a(x,y)\nabla_{x} \overline{u}_{j}(x) \overline{u}_{\ell}(y) u_{j}(x) \Delta_{y} u_{\ell}(y) \,dxdy
\nonumber\\
4 \Re \int_{\mathbb{R}^{2d}} \nabla_{y} a(x,y)\nabla_{y} \overline{u}_{\ell}(y) \overline{u}_{j}(x) u_{\ell}(y) \Delta_{x} u_{j}(x) \,dxdy
\nonumber\\
+4 \Re \int_{\mathbb{R}^{2d}}\nabla_{y} a(x,y)\nabla_{y} \overline{u}_{\ell}(y) \Delta_{y} u_{\ell}(y)|u_{j}(x)|^{2}  \,dxdy 
\nonumber\\
=A_{1}+A_{2}+A_{3}+A_{4}.
\end{align}
It is crucial to observe that
$$
A_{2}=-2 \Re \int_{\mathbb{R}^{2d}} \Delta_{x} a(x,y) \overline{u}_{\ell}(y) \Delta_{y} u_{\ell}(y)|u_{j}(x)|^{2} \,dxdy
$$
and
$$
A_{3}=-2 \Re \int_{\mathbb{R}^{2d}} \Delta_{y} a(x,y) \overline{u}_{j}(x) \Delta_{x} u_{j}(x)|u_{\ell}(y)|^{2} \,dxdy.
$$
Then we achieve that
\begin{align}\label{eq.MorEquiv4A}
\eqref{eq.MorEquiv4}+A_{1}+A_{2}+A_{3}+A_{4} \\
 =2\Re \int_{\mathbb{R}^{2d}} (\Delta_{x} a(x,y) \overline{u}_{j}(x)+2\nabla_{x} a(x,y)\nabla_{x} \overline{u}_{j}(x) ) \Delta_{x} u_{j}(x)|u_{\ell}(y)|^{2} \,dxdy 
\nonumber\\
+2\Re \int_{\mathbb{R}^{2d}} (\Delta_{x,y} a(x,y) \overline{u}_{\ell}(y)+2\nabla_{y} a(x,y)\nabla_{y} \overline{u}_{\ell}(y)) \Delta_{y} u_{\ell}(y)|u_{j}(x)|^{2}   \,dxdy
\nonumber\\
 =\int_{\R^{2d}}
 \Delta^2_xa(x,y) | u_j(x)|^2|u_{\ell}(y)|^2\,dxdy+\int_{\R^{2d}}
\Delta^2_ya(x,y) | u_j(x)|^2|u_{\ell}(y)|^2\,dxdy&
\nonumber\\
    -4 \int_{\R^{2d}}  \nabla_{x}u_{j}(x)D^2_{xy}a(x,y)\nabla_{x}\overline u_{j}(x)|u_{\ell}(y)|^2\,dxdy&
 \nonumber\\
 -4\int_{\R^{2d}} \nabla_{y}u_{\ell}(y)D^2_{xy}a(x,y)\nabla_{y}\overline u_{\ell}(y)|u_{\ell}(x)|^2\,dxdy,&
  \nonumber
\end{align}
where the last inequality arises from standard calculations (see \cite{TarVenk}, for example).
Moreover, we have
\begin{align}\label{eq.MorEquiv5}
4 \Re \int_{\mathbb{R}^{2d}} \nabla_{y} a(x,y)\nabla_{y} \overline{u}_{\ell}(y) \overline{u}_{j}(x) u_{\ell}(y) \Delta_{x} u_{j}(x) \,dxdy
\\
=-4 \Re \int_{\mathbb{R}^{2d}} D^2_{xy} a(x,y)u_{\ell}(y) \nabla_{y} \overline{u}_{\ell}(y) \overline{u}_{j}(x) \nabla_{x} u_{j}(x) \,dxdy
\nonumber\\
-4 \Re \int_{\mathbb{R}^{2d}}  \nabla_{y} a(x,y)u_{\ell}(y) \nabla_{y} \overline{u}_{\ell}(y)  |\nabla_{x} u_{j}(x)|^{2} \,dxdy
\nonumber\\
=-4 \Re \int_{\mathbb{R}^{2d}}  \nabla_{y} a(x,y)u_{\ell}(y) \nabla_{y} \overline{u}_{\ell}(y)  |\nabla_{x} u_{j}(x)|^{2} \,dxdy
\nonumber\\
-4  \int_{\mathbb{R}^{2d}} D^2_{xy} a(x,y)\Re(\overline{u}_{\ell}(y) \nabla_{y} u_{\ell}(y))\Re( \overline{u}_{j}(x) \nabla_{x} u_{j}(x)) \,dxdy
\nonumber\\
-4\int_{\mathbb{R}^{2d}} D^2_{xy} a(x,y)\Im(\overline{u}_{\ell}(y) \nabla_{y} u_{\ell}(y))\Im( \overline{u}_{j}(x) \nabla_{x} u_{j}(x)), \,dxdy
\nonumber
\end{align}
by applying $\Im(\overline{u}_{\ell}(y) \nabla_{y} u_{\ell}(y))=-\Im(u_{\ell}(y) \nabla_{y} \overline{u}_{\ell}(y))$ and  
$\Re(B_{1}B_{2})=\Re(B_{1})\Re(B_{2})-\Im(B_{1})\Im(B_{2})$. This last identity enhances also to
\begin{align}\label{eq.MorEquiv6}
2\Re \int_{\mathbb{R}^{2d}} \Delta_{y} a(x,y) \overline{u}_{j}(x) \Delta_{x} u_{j}(x)|u_{\ell}(y)|^{2} \,dxdy
\\
=2\Re \int_{\mathbb{R}^{2d}} D^2_{xy} a(x,y) \nabla_{y}|u_{\ell}(y)|^{2} \overline{u}_{j}(x)\nabla_{x} u_{j} (x)\,dxdy
\nonumber\\
+2 \Re \int_{\mathbb{R}^{2d}}  \nabla_{y} a(x,y) \nabla_{y}|u_{\ell}(y)|^{2}  |\nabla_{x} u_{j}(x)|^{2} \,dxdy
\nonumber\\
=4 \int_{\mathbb{R}^{2d}} D^2_{xy} a(x,y)\Re (\overline{u}_{\ell}(y)\nabla_{y}u_{\ell}(y)) \Re( \overline{u}_{j}(x)\nabla_{x} u_{j} (x)) \,dxdy
\nonumber\\
+4 \Re \int_{\mathbb{R}^{2d}}  \nabla_{y} a(x,y) \Re (\overline{u}_{\ell}(y)\nabla_{y}u_{\ell}(y))  |\nabla_{x} u_{j}(x)|^{2} \,dxdy.
\nonumber
\end{align}
Finally we come to
\begin{align}
\eqref{eq.MorEquiv5}+\eqref{eq.MorEquiv6}
\\=-4\int_{\mathbb{R}^{2d}} D^2_{xy} a(x,y)\Im(\overline{u}_{\ell}(y) \nabla_{y} u_{\ell}(y))\Im( \overline{u}_{j}(x) \nabla_{x} u_{j}(x)) \,dxdy
\nonumber
\end{align}
and  by symmetry to
\begin{align}
2\Re \int_{\mathbb{R}^{2d}} \Delta_{x} a(x,y) \overline{u}_{\ell}(y) \Delta_{y}u_{\ell}(y)|u_{j}(x)|^{2} \,dxdy\\
4 \Re \int_{\mathbb{R}^{2d}} \nabla_{x} a(x,y)\nabla_{x} \overline{u}_{j}(x) \overline{u}_{\ell}(y) u_{j}(x) \Delta_{y} u_{\ell}(y) \,dxdy
\nonumber
\\=-4\int_{\mathbb{R}^{2d}} D^2_{xy} a(x,y)\Im(\overline{u}_{\ell}(y) \nabla_{y} u_{\ell}(y))\Im( \overline{u}_{j}(x) \nabla_{x} u_{j}(x)) \,dxdy.
\nonumber
\end{align}
It follows that
\begin{align}\label{eq.MorEquiv7}
\eqref{eq.MorEquiv4}+A_{1}+A_{2}+A_{3}+A_{4}\\=\eqref{eq.MorEquiv4A}
-8\int_{\mathbb{R}^{2d}} D^2_{xy} a(x,y)\Im(\overline{u}_{\ell}(y) \nabla_{y} u_{\ell}(y))\Im( \overline{u}_{j}(x) \nabla_{x} u_{j}(x)) \,dxdy.
\nonumber
\end{align}
The proof is now  completed.
 \end{proof}

\begin{remark}
We presented our argument only for the terms involving the operator $\Delta_{x}$. Similar calculations can be performed also
for the fourth-order operator  $\Delta^{2}_{x}$, with more complicated steps involved (see \cite{Ta}). It should be not surprising that the above lemma shows that
$$
\int_{\mathbb{R}^{2d}} \Im(\overline{u}_{\ell}(y) \nabla_{y} u_{\ell}(y))D^2_{xy} a(x,y)\Im( \overline{u}_{j}(x) \nabla_{x} u_{j}(x)) \,dxdy=0.
$$
This because the tensor Morawetz identities act as classical Morawetz ones and is not taking into account,
at least in the fundamental steps, of the interactive aspect of the action \eqref{eq:tenmor1}. In fact if $u_{j}, u_{\ell}$ be solutions to \eqref{eq:HFC4} in $d$ spatial dimensions we can define the tensor product $\left(u_{i} \otimes u_{\ell}\right)(t, x,y)$ for $(x,y)$ in
$$
\mathbb{R}^{2d}=\left\{\left(x, y\right): x \in \mathbb{R}^{d}, y \in \mathbb{R}^{d}\right\},
$$
by the formula
$$
\left(u_{j} \otimes u_{\ell}\right)(t, x,y)=u_{j}\left(t, x\right) u_{\ell}\left(t, y\right)
$$
and utilizing \eqref{eq.singMor} along with the equation  \eqref{eq.tenid}. The interactive inequalities are then a byproduct of this approach. At this point there are several ways one can present these estimates: as bilinear generalization of the classical Morawetz estimates 
(see \cite{PlVe}, \cite{TarVenk},  \cite{TzVi}) or as classical Morawetz estimates applied to tensors of solutions  (see \cite{Dinh} \cite{MWZ}, \cite{Ta}). 
\end{remark}


\begin{thebibliography}{CTCR81}

%\bibitem{BarceloRuizVega} J.A.Barcelथक, A.Ruiz, and L.Vega, \emph{Some dispersive estimates for
%Schr\"odinger equations with repulsive potentials},
  %J. Funct. Anal. 236 (2006), 1ै��24.
  

\bibitem{BKS00} M. Ben-Artzi, H. Koch, and J. C. Saut, \emph{Dispersion estimates for fourth order Schr\"odinger equations}, C.R.A.S. 330(1), 2000, pp.~87--92.

\bibitem{BJPSS} N. Benedikter, V. Jaksic, M. Porta, C. Saffirio, B. Schlein, \emph{Mean-field Evolution of Fermionic Mixed States}, Comm. Pure Appl. Math. 69, no. 12, 2016, pp.~2250--2303.

\bibitem{BSS2} N. Benedikter, J. Sok, and J. P. Solovej,  \emph{The Dirac-Frenkel Principle for Reduced
Density Matrices, and the Bogoliubov-de-Gennes Equations}, Ann. Henri Poincar\'e 19, 2018, pp.~1167--1214.

\bibitem{CM} R. Carles, E. Moulay,  \emph{High order Schr{\"o}dinger equations}, J. Phys. A: Math. Theor. 45, 2012.

\bibitem{CLM} R. Carles, W. Lucha, E. Moulay, \emph{High order Schr{\"o}dinger and Hartree-Fock equations}, J. Math. Phys 56, 12, 2015, pp.~122--301.

\bibitem{ChS} P. Choquard, J. Stubbe, \emph{The one-dimensional Schr{\"o}dinger-Newton equations}, Lett. Math. Phys. 81, 2007, pp.~177--184.


\bibitem{CSV} P. Choquard, J. Stubbe, M Vuffray,  \emph{Stationary solutions of the Schr{\"o}dinger- Newton model-an ODE approach}, Differ. Integral Equ. 21, no. 7-8, 2008, pp.~665-679. 

\bibitem{CT} B. Cassano, M. Tarulli, 
\emph{$H^1$-scattering for systems of $N$-defocusing weakly coupled NLS equations in low space dimensions}. J. Math. Anal. Appl. 430, 2015, pp.~528--548.


\bibitem{CGMZ} J. Chong, M. Grillakis, M. Machedon,  Z. Zhao, \emph{Global estimates for the Hartree-Fock-Bogoliubov equations}, Communications in Partial Differential Equations, 2021. DOI: 10.1080/03605302.2021.1920615.

\bibitem{Ca} T. Cazenave, \emph{Semilinear Schr\"odinger equations}, Courant Lecture Notes in Mathematics, 10, New York University Courant Institute of Mathematical Sciences, New York, 2003.

\bibitem{CGT07} J. Colliander, M. Grillakis, N. Tzirakis, \emph{Improved interaction Morawetz inequalities for the cubic nonlinear Schr\"odinger equation on $\R^{2}$}. Int. Math. Res. Not., 23, 2007.

\bibitem{CGT}  J. Colliander, M. Grillakis, N. Tzirakis \emph{Tensor products and correlation estimates with applications to nonlinear Schr\"odinger equations}, Comm. Pure Appl. Math., 62(7) , 2009, pp.~920--968.
 

\bibitem{Dinh} V. D. Dinh, \emph{On Nonlinear Schr\"odinger Equations with Repulsive Inverse-Power Potentials}, Acta Appl. Math. (2021). https://doi.org/10.1007/s10440-020-00382-2.

 \bibitem{ES} A. Elgart, B. Schlein, \emph{Mean field dynamics of boson stars}, Comm. Pure Appl. Math. 60(4), 2007, pp.~500--545.


\bibitem{FWY} H. Feng, H. Wang, X. H. Yao, \emph{Scattering Theory for the Defocusing Fourth Order NLS with Potentials}
Acta Mathematica Sinica, V. 34, No. 4, 2018, pp.~773--786.


\bibitem{FIP} G. Fibich, B. Ilan, G. Papanicolaou, \emph{Self-focusing with fourth order dispersion}, SIAM J. Appl. Math.
62(4), 2002, pp.~1437--1462.

\bibitem{Fo} V. A. Fock, \emph{N{\"a}herungsmethode zur L{\"o}sung des quantenmechanischen Mehrk{\"o}rperproblems}, Zeit. f{\"u}r Physik 61(1-2), 1930126-148.

\bibitem{FrLe} J. Fr{\"o}hlich, E. Lenzmann, \emph{Dynamical collapse of white dwarfs in Hartree-and Hartree-Fock theory}, Commun. Math. Phys. 274, 2007, pp.~737--750.

\bibitem{FM00} Y. Fukomoto, H. Moffatt, \emph{Motion and expansion of a viscous vortex ring. Part 1. A higher-order asymptotic formula for the velocity}, Journal of Fluid Mechanics, 417, 2000, pp.~1--45.

\bibitem{Guo} Q. Guo, \emph{Scattering for the focusing $L^{2}$-supercritical and $H^{2}$-subcritical bi-harmonic NLS
equations}, Commun. Partial Differ. Equ., 41 2016, pp.~185--207.

\bibitem{HJ05} Z.  Huo, Y. Jia, \emph{The Cauchy problem for the fourth-order nonlinear Schr\"odinger equation related to the vortex filament}, J. Diff. Equ., 214, 2005, pp.~1--35.

\bibitem{HJ07} Z.  Huo, Y. Jia, \emph{A refined well-posedness for the fourth-order nonlinear Schr\"odinger equation related to the vortex filament}, Commun. Partial Diff. Eqns., 32, 2007, pp.~493--510.




\bibitem{Kar} V.  I.  Karpman, \emph{Stabilization of soliton instabilities by higher-order dispersion: Fourth order nonlinear Schr\"odinger-type equations}, Phys. Rev. E, 53(2), 1996, pp.~1336--1339. 

\bibitem{KaS} V.  I.  Karpman, A. G.  Shagalov, \emph{Stability of soliton described by nonlinear Schr\"odinger-type equations with higher-order dispersion}, Phys. D 144, (2000) 194-210.

\bibitem{Kat} T.~Kato \emph{On nonlinear Schr\"odinger equations: II. $H^{s}$-solutions and unconditional well-posedness}, J. Anal. Math., 67 1995, pp.~281-306.

\bibitem{Len} E. Lenzmann, \emph{Well-posedness for semi-relativistic Hartree equations of critical type}, Math. Phys. Anal. Geom. 10, 1, 2007, pp.~43-64.

\bibitem{LS} S. Levandosky, W. Strauss, \emph{Time decay for the nonlinear Beam equation}, Methods Appl. Anal. 7, (2000), 479-488.

 \bibitem{Lip} E. Lipparini, Modern Many-Particle Physics: Atomic Gases, Nanostructures and Quantum Liquids, 2nd ed. World Scientific Publishing Company, (2008).



\bibitem{LR} M. Lewin, N. Rougerie, \emph{Derivation of Pekar's polarons from a microscopic model of quantum crystal}, SIAM J. Math. Anal. 45, no. 3, 1267-1301 (2013).



\bibitem{MWZ} C. Miao, H. Wu, and J. Zhang, \emph{Scattering theory below energy for the cubic fourth-order Schr\"odinger equation},  Math. Nachr. 288, No. 7, 2015, pp.~798--823.

 

\bibitem{Pa1} B.  Pausader \emph{Global well-posedness for energy critical fourth-order Schr\"odinger equations in the radial case} Dyn. Partial Diff. Eqns, 4, 2007, pp.~197--225.


\bibitem{PaSh} B. Pausader, S. Shao  \emph{The mass-critical fourth-order Schr\"odinger equation in high dimensions}, J. Hyp.
Diff. Eqns, 7, 2010, pp.~651--705.

\bibitem{PaX} B. Pausader, S. Xia, \emph{Scattering theory for the fourth-order Schr\"odinger equation in low dimensions}
Nonlinearity 26, 2013, pp.~2175--2191.

\bibitem{Pen} R. Penrose, \emph{Quantum computation, entanglement and state reduction}, Phil. Trans. R. Soc.,
356, no. 1743, 1998, pp.~1927--1939.

\bibitem{PlVe} F. Planchon, L. Vega, \emph{Bilinear virial identities and applications}, Ann. Sci. \'Ec. Norm. Sup\'er. (4) 42, 2009, pp.~261--290.


\bibitem{Ta} M. Tarulli, \emph{$H^2$-scattering for Systems of Weakly Coupled Fourth-order NLS Equations in Low Space Dimensions}, Potential Analysis, available online 2018, 51(2), 2019, pp.~291--313. 

\bibitem{TarVenk} M. Tarulli, G. Venkov, \emph{Decay and scattering in energy space for the solution of weakly coupled Schrodinger-Choquard and Hartree-Fock equations}. J. Evol. Equ., 2020. DOI: 10.1007/s00028-020-00621-x.

\bibitem{TeTzVi} S. Terracini, N. Tzvetkov, N. Visciglia, \emph{The Nonlinear Schr{\"o}dinger equation ground states on product spaces}, Analysis \& PDE 7, no. 1, 73-96 (2014).

\bibitem{TzVi} N. Tzvetkov, N. Visciglia, \emph{Well-posedness and scattering for NLS on $\R^d\times \T$ in the energy space}, Rev. Mat. Iberoam., 32(4), 2016, pp.~1163--1188.

\bibitem{Saa1} T. Sannouni, \emph{Non-linear Bi-harminic choquard equations}, Comm. Pure and Appl. Analysis, 19(11), 2020, pp.~5033--5057.

\bibitem{Saa2} T. Sannouni, \emph{Scattering for Radial Defocusing Inhomogeneous Bi-Harmonic Schr\"odinger Equations}, Potential Analysis, 2020. https://doi.org/10.1007/s11118-020-09898-6.

\bibitem{Se03} J. Segata, \emph{Well-posedness for the fourth-order nonlinear Schr\"odinger type equation related to the vortex filament}, Diff. Int. Eqns., 16, 2003, pp.~841--64

\bibitem{Se06} J. Segata, \emph{Modified wave operators for the fourth-order non-linear Schr\"odinger-type equation with cubic non-linearity}, Math. Methods Appl. Sci., 26, 2006, pp.~1785--1800.

\bibitem{Vis} N.~Visciglia, \emph{On the decay of solutions to a class of defocusing NLS}, Math. Res. Lett. 16(5) 2009, pp.~919--926.







\end{thebibliography}
\end{document}